\documentclass[12pt,reqno]{amsart}

\usepackage{amsfonts,color,amsthm,amsmath,amssymb}

\usepackage{color}
\def\Box{\vcenter{\vbox{\hrule\hbox{\vrule
     \vbox to 8.8pt{\hbox to 10pt{}\vfill}\vrule}\hrule}}}

\def\qed{{\hfill$\square$}}
\def\proof{{\vspace{-0.0cm}\bf Proof: \,}}

\def\Z{{\mathbb Z}}

\def\T{{\mathrm{Tr}}}

\def\F{{\mathbb F}}

\def\mod{{\mathrm{mod\,\,}}}

\def\Tr{{\mathrm{Tr}}}
\def\Norm{{\mathrm{Norm}}}

\def\tra{{\mathrm{T}}}

\newtheorem{theorem}{Theorem}[section]
\newtheorem{problem}[theorem]{Problem}
\newtheorem{lemma}[theorem]{Lemma}
\newtheorem{remark}[theorem]{Remark}

\newtheorem{proposition}[theorem]{Proposition}
\newtheorem{example}[theorem]{Example}

\numberwithin{equation}{section}

\begin{document}
\title[New families of  Hadamard matrices with maximum excess]
{New families of Hadamard matrices with maximum excess}

\author[Mitsugu Hirasaka, Koji Momihara and Sho Suda]{Mitsugu Hirasaka$^\ast$, Koji Momihara$^{\dagger}$ and  Sho Suda$^{\ddagger}$}

\thanks{$^\ast$ Mitsugu Hirasaka is supported by Basic Science Research 
Program through the National Research Foundation of Korea (NRF) funded 
by Ministry of Science, ICT \& Future Planning (NRF-2016R1A2B4013474).}
\thanks{$^{\dagger}$ Koji Momihara is supported by JSPS KAKENHI Grant Number 17K14236 and 15H03636.}
\thanks{$^{\ddagger}$
Sho Suda is supported by  JSPS KAKENHI Grant Number 15K21075.}

\address{$^\ast$ Department of Mathematics, Pusan National University, Busan 609-735, Republic of Korea} \email{hirasaka@pusan.ac.kr}

\address{$^\dagger$ Faculty of Education, Kumamoto University, 2-40-1 Kurokami, Kumamoto 860-8555, Japan} \email{momihara@educ.kumamoto-u.ac.jp}

\address{$^{\ddagger}$ Department of Mathematics Education, Aichi University of 
Education, 1 Hirosawa, Igaya-cho, Kariya, Aichi 448-8542, Japan} \email{suda@auecc.aichi-edu.ac.jp}


\begin{abstract}
In this paper, we find regular or biregular Hadamard matrices with maximum excess by negating some rows and columns of known 
Hadamard matrices obtained from quadratic residues of finite fields. 
In particular, we show that if either $4m^2+4m+3$ or  $2m^2+2m+1$ is a prime power, then there exists a biregular Hadamard 
matrix of order $n=4(m^2+m+1)$ with maximum excess. Furthermore, we  give a sufficient condition for Hadamard matrices obtained from quadratic residues being transformed to be regular in terms of four-class
translation association schemes on finite fields. 
\end{abstract}

\maketitle

\section{Introduction}\label{sec:int}
An  {\it Hadamard matrix of order $n$} is an $n\times n$ $(-1,1)$-matrix $H$  satisfying $HH^\tra=H^\tra H=nI_n$, where $I_n$ is the $n\times n$ unit matrix. 
It is well known that $n=1,2$ or a multiple of $4$. 
Conversely, it is conjectured that an Hadamard matrix of order $n$ exists for every positive integer $n$ divisible by $4$. 

It is clear that if rows and columns of an Hadamard matrix  $H$ are permuted, 
the matrix remains to be an Hadamard matrix. 
Furthermore, the multiplication of any row and column of $H$ by $-1$ also remains to be an Hadamard matrix. Thus, it is always possible to transform $H$ to have the first row and the first column contain  only $+1$ entries. 
Such an Hadamard matrix is said to be {\it normalized}. We say that two Hadamard matrices are {\it equivalent} if one can be obtained from the other 
by  a sequence of row and column permutations and negations. This is same as saying that two Hadamard matrices are equivalent if one can obtained from the other by premultiplication and postmultiplication by signed permutation matrices. 
In this paper, we mainly treat negations of rows and columns of Hadamard matrices. In this case, the corresponding signed permutation matrices  are just $(1,-1)$-diagonal matrices. 
 
Let ${\bf 1}_n$ denote the all-one vector of length $n$. 
An Hadamard matrix $H$ is called {\it regular} if $H{\bf 1}_n=r{\bf 1}_n$ 
for some positive integer $r$. 
Write $H{\bf 1}_n=(h_1,h_2,\ldots,h_n)^\tra$. 
Then, 
\[
nr^2=\sum_{i=1}^nh_i^2=({\bf 1}_n^\tra H^{\tra})(H{\bf 1}_n)
=n{\bf 1}_n^\tra{\bf 1}_n=n^2. 
\] 
Hence, 
$n$ must be 
a square, namely, $n=4m^2$, and $r=2m$. 

We say that an Hadamard matrix $H$ is {\it biregular} if the 
entries of $H{\bf 1}_n$ take exactly two nonnegative integers, namely, $k_1$ and $k_2$. 
Here, $k_i$, $i=1,2$, must be even integers in the same residue class modulo $4$ from the orthogonality of the rows of $H$. Let $m_1$ (resp. $m_2$) be the frequency of $k_1$ (resp. $k_2$) appearing in $H{\bf 1}_n$.  Write 
$H{\bf 1}_n=(h_1,h_2,\ldots,h_n)^\tra$. 
Then, 
\[
m_1+m_2=n
\]
and 
\[
m_1k_1^2+m_2 k_2^2=\sum_{i=1}^nh_i^2=({\bf 1}_n^\tra H^{\tra})(H{\bf 1}_n)
=n{\bf 1}_n^\tra{\bf 1}_n=n^2. 
\] 
Hence, 
\begin{equation}\label{eq:fre}
m_1=\frac{n^2-nk_2^2}{k_1^2-k_2^2}\hspace{0.4cm} \mbox{and }\hspace{0.2cm} m_2=\frac{-n^2+nk_1^2}{k_1^2-k_2^2}. 
\end{equation}
These classes of Hadamard matrices produce block designs with intersecting properties explained as in the next section. In this paper, we are interested in which Hadamard matrices can be transformed to be regular or biregular. 


We now explain one of major motivations for studying regular or biregular 
Hadamard matrices. 
Let $E(H)$ denote the sum of all entries of $H$. We say that $E(H)$ is the {\it excess} of $H$. Several upper bounds on $E(H)$ have been known~\cite{B,HLS,KF}. We will give one of the known upper bounds on $E(H)$ below. 
\begin{proposition} {\em (\cite{HLS})} \label{prop:Hadabound}
Let $H$ be an Hadamard matrix of order $n$ and let $k$ be an even integer such that  $k\le \sqrt{n}<k+2$. Put $t=k$ if $|n-k^2|<|n-(k+2)^2|$ and $t=k-2$ otherwise. 
Then, 
it holds that 
$E(H)\leq n(t+4)-4s$, where $s$ is the integer part of $n((t+4)^2-n)/(8t+16)$, 
with equality holds if and only if $n$ is a square and $H$ is regular or 
$n$ is a nonsquare and 
 the entries of $H{\bf 1}_n$ are in either 
$\{k,k+4\}$ or $\{k-2,k+2\}$ depending on whether  $t=k$  or $t=k-2$, respectively. 
\end{proposition}
Note that if an Hadamard matrix $H$ is regular or biregular with maximum excess attaining the bound of 
Proposition~\ref{prop:Hadabound}, then $H^\tra$ is also regular or biregular, respectively. 

The excess of Hadamard matrices have been studied in \cite{B,CK,FK,HKL,HK,K,KK,KS,KF,S,XXS}. 
In particular, many constructions of regular Hadamard matrices have been 
known in relation to {\it Hadamard designs}. In fact, any Hadamard design yields a regular Hadamard matrix. 
On the other hand, as far as the authors know, there is only a few  paper theoretically treating biregular Hadamard matrices with maximum excess~\cite{K,KKS}. In particular, in \cite{KKS}, it was shown that 
there is a biregular Hadamard matrix of order $n=4(m^2+m+1)$ with maximum excess attaining the bound of Proposition~\ref{prop:Hadabound} if $m$ is a prime power and $m^2+m+1$ is a prime.  In this paper, we will also treat 
Hadamard matrices of order $n=4(m^2+m+1)$. 

In this paper, we give three constructions of regular or biregular Hadamard matrices. We briefly explain the constructions. 

Let $\F_q$ be the finite field of order $q$ and $C$ be the set of nonzero squares of $\F_{q}$. 
We start from the following known construction of Hadamard matrices. 
Assume that $q\equiv 3\,(\mod{4})$. Let $M$ be a $q\times q$ $(1,-1)$-matrix whose rows and columns are labeled by the elements of $\F_q$ and entries are defined by 
\[
M_{i,j}=\begin{cases}
1,& \text{ if } j-i\in C \cup \{0\},\\
-1,& \text{ if } j-i\in \F_q\setminus (C \cup \{0\}). 
\end{cases} 
\]
Define 
\begin{equation}\label{eq:confee}
H=\begin{pmatrix} -1  & {\bf 1}_q^\tra \\ {\bf 1}_q & M \end{pmatrix}. 
\end{equation}
Then $H$ forms an Hadamard matrix of order $n=q+1$. We will prove the following theorem in Section~\ref{sec:const1}. 
\begin{theorem}\label{thm:main1}
If $q=4m^2+4m+3$ is a prime power, then there exists a biregular Hadamard 
matrix of order $n=4(m^2+m+1)$ with maximum excess attaining the bound of 
Proposition~\ref{prop:Hadabound}. In particular, the  matrix $H$
defined in \eqref{eq:confee} can be 
transformed to be a biregular Hadamard matrix with maximum excess by negating some rows and columns of $H$. 
\end{theorem}

Next, we consider the case where $q\equiv 1\,(\mod{4})$. 
Let $M$ be a $q\times q$ $(0,1,-1)$-matrix whose rows and columns are labeled by the elements of $\F_q$ and entries are defined by 
\[
M_{i,j}=\begin{cases}
0,& \text{ if } j-i=0,\\
1,& \text{ if } j-i\in C,\\
-1,& \text{ if } j-i\in \F_q \setminus (C \cup \{0\}). 
\end{cases} 
\]
Define $M_1=M+I_q$, $M_2=M-I_q$, and $M_3=-M_1$. Furthermore, 
define 
\begin{equation}\label{eq:hada2ee}
H=\begin{pmatrix} 
1  &-1 &  {\bf 1}_q^\tra&  {\bf 1}_q^\tra  \\
-1  & -1 & {\bf 1}_q^\tra&-{\bf 1}_q^\tra \\
{\bf 1}_q  &  {\bf 1}_q &M_1 & M_2 \\
{\bf 1}_q  & -{\bf 1}_q&M_2 & M_3  
 \end{pmatrix}. 
\end{equation}
Then, $H$ forms an Hadamard matrix of order $n=2q+2$. 
We will prove the following theorem in Section~\ref{sec:const2}. 
\begin{theorem}\label{thm:main2}
If $q=2m^2+2m+1$ is a prime power, then there exists a biregular Hadamard 
matrix of order $n=4(m^2+m+1)$ with maximum excess attaining the bound of 
Proposition~\ref{prop:Hadabound}. In particular, the Hadamard matrix $H$
defined in \eqref{eq:hada2ee} can be 
transformed to be a biregular Hadamard matrix with maximum excess by negating some rows and columns of $H$.  
\end{theorem}

We also show the following theorem to obtain a regular Hadamard matrix from the matrix $H $ defined in \eqref{eq:hada2ee}. 
\begin{theorem}\label{thm:main3tr}
Let  $q=2m^2-1$ be a prime power with $m$ odd. 
Let $X_0=\{0\}$, and  assume that there are subsets $X_i$, $i=1,2,3,4$, of $\F_{q^2}$ partitioning $\F_{q^2}\setminus \{0\}$ satisfying the following conditions: 
\begin{enumerate}
\item[(1)] $X_1=\omega^{2m^2}X_3$ and  $X_2=\omega^{2m^2}X_4$, where $\omega$ is a fixed primitive element of $\F_{q^2}$;
\item[(2)] Each $X_i$,  $i=1,2,3,4$, is a union of cosets of the multiplicative subgroup of index $4m^2$ of $\F_{q^2}$; 
\item[(3)] $(\F_{q^2},\{R_i\}_{i=0}^4)$ is a four-class translation association scheme with a prescribed first eigenmatrix. Here, for $i=0,1,2,3,4$, $(x,y)\in  R_i$ if and only if $x-y \in X_i$. 
\end{enumerate}
Then, there exists a regular Hadamard 
matrix of order $n=4m^2$. In particular, the Hadamard matrix $H$
defined in \eqref{eq:hada2ee} can be 
transformed to be a regular Hadamard matrix by negating some rows and columns of $H$.  
\end{theorem}
We will explain this theorem in details in Section~\ref{sec:const3}. This theorem implies that  a four-class translation association scheme on $\F_{q^2}$ produces a regular Hadamard 
matrix of order $n=2(q+1)$. Thus, we find a nontrivial relationship between 
regular Hadamard matrices and association schemes. 

\section{Preliminaries}
\subsection{Block designs and $t$-intersection sets}
In this subsection, we assume that the reader is familiar  with the basic theory of block designs. We refer the reader to \cite{BJL,J}. 

Let $P$ be a finite set of $v$ elements, called {\it points}, and ${\mathcal B}$ be a family of $b$ subsets of $P$, called {\it blocks}. We define ${\mathcal F}=\{(p,B)\in P\times {\mathcal B}\,:\,p\in B\}$. Elements in ${\mathcal F}$ are called {\it flags}.  The triple 
$(P,{\mathcal B},{\mathcal F})$ is called a {\it block design}. 
We say that $({\mathcal B},P,{\mathcal F}^\perp)$ with 
${\mathcal F}^\perp=\{(B,p)\,:\,(p,B)\in {\mathcal F}\}$ is the {\it dual} of $(P,{\mathcal B},{\mathcal F})$. 
For convenience, we also say that the pair $(P,{\mathcal B})$ is  a {\it block design}. 
We denote by $M$ the incidence matrix of a block design $(P,{\mathcal B})$, whose $(p,B)$th entry is $1$ if  $(p,B)\in {\mathcal F}$ (equivalently, $P\in B$) and $0$ otherwise. 

A block design $(P,{\mathcal B})$ with incidence matrix $M$ is called a {\it pairwise balanced design} if there is a positive integer $\lambda$ such that 
each off-diagonal entry  of  $MM^\tra$ is to $\lambda$. In addition, if $M {\bf 1}_b=r {\bf 1}_v$ for some integer $r$, it is called an {\it $(r,\lambda)$-design}. 
Furthermore, if an  $(r,\lambda)$-design satisfies that ${\bf 1}_v^\tra M=k {\bf 1}_b$, it is called a {\it $2$-design}.  In particular, if $v=b$, it is called 
{\it symmetric}.  

Let $H$ be a regular Hadamard matrix of order $n=4m^2$. Then, the $(0,1)$-matrix $M=(-H+J_n)/2$ satisfies that 
\[
MM^\tra=\frac{HH^\tra-HJ_{n}-J_{n}H^\tra+J_{n}J_{n}}{4}=m^2I_{n}+(m^2-m)J_{n}, 
\]
where $J_n$ is the $n\times n$ all-one matrix. 
This implies that 
$M$ is an incidence matrix of a 
symmetric  $2$-design with parameters $(v,k,\lambda)=(4m^2,2m^2-m,m^2-m)$, 
the so-called {\it Hadamard design}. 

Let $H$ be a biregular Hadamard matrix of order $n$.  We can assume that 
 \begin{align*}
H{\bf 1}_n=\begin{pmatrix} k_1 {\bf 1}_{m_1}\\ k_2 {\bf 1}_{m_2} \end{pmatrix}
\end{align*}
by suitably permuting rows.  
Let $H_1$ (resp. $H_2$) be the upper $m_1\times n$ (resp. lower $m_2\times n$) matrix of $H$. Then, each
$M_i=(-H_i+J_{m_i,n})/2$, $i=1,2$, satisfies that  
\[
M_iM_i^\tra=\frac{nH_iH_i^\tra-H_iJ_{m_i,n}^\tra-J_{m_i,n}H_i^\tra+J_{m_i,n}J_{m_i,n}^\tra}{4}=\frac{nI_{m_i}+(n-2k_i)J_{m_i}}{4}, 
\]
where $J_{m_i,n}$ is the $m_i\times n$ all-one matrix. 
Hence, each $M_i$  is the incidence matrix of a pairwise balanced design. In particular, each $M_i$ satisfies that $M_i{\bf 1}_n=(n-k_i)/2$. Hence, it  is  an $((n-k_i)/2,(n-2k_i)/4)$-design. 
Furthermore, $M_i$'s satisfy that 
\[
M_1M_2^\tra=\frac{H_1H_2^\tra-H_1J_{m_2,n}^\tra-J_{m_1,n}H_2^\tra+J_{m_1,n}J_{m_2,n}^\tra}{4}=\frac{(n-k_1-k_2)J_{m_1,m_2}}{4}. 
\]
Thus, biregular 
Hadamard matrices also have interesting properties 
in view of design theory. 

We introduce the concept of $t$-intersection sets for block designs.  
We refer the reader to \cite{MS} for details of two-intersection sets. 
Let $(P,{\mathcal B})$ be a block design with incidence matrix $M$ and  $D$ be a $j$-subset of $P$. We say that $D$ is a 
{\it $t$-intersection set with parameters $(j;\{\alpha_1,\alpha_2,\ldots,\alpha_t\})$ for $(P,{\mathcal B})$} if  
\[
\{|B\cap D|\,:\,B\in {\mathcal B}\}=\{\alpha_1,\alpha_2,\ldots,\alpha_t\}. 
\]
Let ${\bf x}$ be the $(0,1)$-vector of length $v$, whose $i$th entry is $1$ if $i\in D$ and $0$ otherwise. We call ${\bf x}$ the {\it support vector} of $D$ in $P$. 
Then, $D$ is a $t$-intersection set with parameters 
$(j;\{\alpha_1,\alpha_2,\ldots,\alpha_t\})$ for $(P,{\mathcal B})$ if and only if 
all entries of ${\bf x}^\tra M$ are in  $\{\alpha_1,\alpha_2,\ldots,\alpha_t\}$ and ${\bf x}^\tra {\bf 1}_{v}=j$.  
We define the {\it duals} of $D$ as 
\begin{equation}\label{eq:defdual}
D_{\alpha_i}^\perp=\{B\,:\,B\in {\mathcal B},\,|B\cap D|=\alpha_i\}, i=1,2,\ldots,t. 
\end{equation}
In this paper, we treat $t$-intersection sets with $t\le 4$ for a block design 
obtained from translations of the set of squares in the finite field. In particular, we will transform a (non-regular) Hadamard matrix to a regular  or biregular Hadamard matrix by negating columns and rows  corresponding to a  $t$-intersection set and its dual, respectively.

\subsection{Association schemes}
In this subsection, we give a short introduction to association schemes. 
We refer the reader to \cite{BI} for the general theory of association schemes. 

Let $X$ be a finite set,  and a set of 
relations  $R_0, R_1,\ldots,R_d$ be a partition 
of $X\times X$ such that $R_0=\{(x,x)\,:\,x\in X\}$, and 
for each $i\in \{0,1,\ldots,d\}$, ${}^tR_i=R_{i'}$ for some $i'\in \{0,1,\ldots,d\}$, where ${}^t R_i=\{(x,y)\in X^2\,:\,(y,x)\in R_i\}$. 
We call $(X,\{R_i\}_{i=0}^d)$ a {\it $d$-class association scheme} if for all $i,j,k\in \{0,1,\ldots,d\}$ there is an integer $p_{i,j}^k$ such that for all $(x,y)\in R_k$, 
\[
|\{z\in X\,|\,(x,z)\in R_i,\,(z,y)\in R_j)\}|=p_{i,j}^k.
\] 
These constants are called {\it intersection numbers}. 
If $p_{i,j}^k=p_{j,i}^k$ for all $h,i,j\in \{0,1,\ldots,d\}$, it is called {\it commutative}. 
If ${}^tR_i=R_{i}$ for all $i\in \{0,1,2,\ldots,d\}$, then  it is called {\it symmetric}. A symmetric association scheme is commutative. In this paper, we will treat symmetric association schemes.  

We denote by $A_i$ the adjacency matrix of $R_i$ for each $i$, whose $(x,y)$th entry is $1$ if $(x,y)\in R_i$ and $0$ otherwise. 
The condition above is equivalent to that 
\[
A_iA_j=\sum_{k=0}^d p_{i,j}^k A_k. 
\]  
The {\it Bose-Mesner algebra} for an association scheme $(X,\{R_i\}_{i=0}^d)$ is defined as ${\mathcal A}=\langle A_0,A_1,\ldots,A_d\rangle$. Since each $R_i$ is symmetric, ${\mathcal A}$ is commutative. Then, there exists a
set of minimal idempotents $E_0=\frac{1}{|X|}J_{d+1},E_1,\ldots,E_d$, which also form a basis of the algebra. The matrix $P$ such that 
\[
(A_0,A_1,\ldots,A_d)=(E_0,E_1,\ldots,E_d)P
\]
is called the {\it first eigenmatrix} (or {\it character table}) of $(X,\{R_i\}_{i=0}^{d})$.  On the other hand, 
the matrix $Q$ such that 
\[
(A_0,A_1,\ldots,A_d)Q=|X|(E_0,E_1,\ldots,E_d)
\]
is called the {\it second eigenmatrix} of $(X,\{R_i\}_{i=0}^{d})$.

A {\it translation association scheme} is an association scheme $(X,\{R_i\}_{i=0}^{d})$ for which the underlying set $X$ is identified with 
an abelian group, and for all relations $R_i$'s, $(x,y)\in R_i$ implies $(x+z,y+z)\in R_i$ for all $z\in X$. Then, there is a partition $D_0=\{0\},D_1,\ldots,D_d$ of $X$  such that for each $i=0,1,\ldots,d$, 
\[
R_i=\{(x,x+y)\,:\,x\in X,y\in D_i\}. 
\]
For a translation association scheme $(X,\{R_i\}_{i=0}^{d})$, there is an 
equivalence relation  defined on the character group $X^\perp$ of $X$ 
as follows: $\chi\sim \chi'$ if and only if $\chi(D_i)=\chi'(D_i)$ for all $i=0,1,\ldots,d$. Here,  $\chi(D_i):=\sum_{x\in D_i}\chi(x)$. Denote by $D_0',D_1'\ldots,D_d'$ the equivalence classes, where $D_0'$ consists of only the trivial character. Define the relation $R_i'$ on $X^\perp$ as  
\[
R_i'=\{(\chi,\chi \chi')\,:\,\chi \in X^\perp,\chi'\in D_i'\}. 
\]  
Then, $(X^\perp,\{R_i'\}_{i=0}^d)$ forms a translation association scheme,
called the {\it dual} of $(X,\{R_i\}_{i=0}^{d})$. We remark that the entries of 
the $i$th column of the  first eigenmatrix of $(X,\{R_i\}_{i=0}^{d})$ are given by 
the $d+1$ character values $\chi(D_i)$, where $\chi\in D_j'$, $j=0,1,2\ldots,d$. 
Furthermore, 
the first eigenmatrix of the dual scheme is equal to the second eigenmatrix of  $(X,\{R_i\}_{i=0}^{d})$. 

Given two association schemes $(X,\{R_i\}_{i=0}^d)$ and $(X,\{R_i'\}_{i=0}^e)$, 
if for each $i=0,1,\ldots,d$,  $R_i\subseteq R_j'$ for some $j$, then 
$(X,\{R_i\}_{i=0}^d)$ is called a {\it fission scheme} of $(X,\{R_i'\}_{i=0}^e)$, and 
$(X,\{R_i'\}_{i=0}^e)$ is called a {\it fusion scheme} of $(X,\{R_i\}_{i=0}^d)$. 
We will use the following well-known criteria due to Bannai~\cite{Ba} and 
Muzychuk~\cite{Mu}, called the {\it Bannai-Muzychuk criterion}. 
 \begin{proposition} \label{prop:BM}
Let $P$ be the first eigenmatrix of an association scheme $(X,\{R_i\}_{i=0}^d)$, and let $\Lambda_0:=\{0\}, \Lambda_1,\ldots,\Lambda_e$, be a partition of $\{0,1,\ldots,d\}$. Then, $(X,\{\bigcup_{j\in \Lambda_i}R_j\}_{i=0}^e)$ forms an association scheme if and only if 
there exists a partition $\{\Delta_0=\{0\},\Delta_1,\ldots,\Delta_e\}$ of 
$\{0,1,\ldots,d\}$ such that $(\Delta_i,\Lambda_j)$-block of $P$ has a 
constant row sum. Moreover, the constant row sum of the  $(\Delta_i,\Lambda_j)$-block is the $(i,j)$th entry of the first eigenmatrix of 
$(X,\{\bigcup_{j\in \Lambda_i}R_j\}_{i=0}^e)$. 
%
\end{proposition}
\subsection{Characters of finite fields}\label{sec:31}
In this subsection, we will assume that the reader is familiar with the basic theory of characters of finite fields.

Let $p$ be a prime, $f$ a positive integer, and set $q=p^f$. For a positive integer $m$, let $\zeta_m=\exp(\frac{2\pi i}{m})$ denote the primitive $m$th root of unity. 
Let $\F_q$ denote the finite field of order $q$ and $\F_q^\ast$ be the multiplicative group of $\F_q$. For a positive integer 
$e$ dividing $q-1$ and a fixed primitive element $\omega$ of $\F_q$, define \[
C_i^{(e,q)}:=\omega^i\langle \omega^e\rangle, \, \, i=0,1,\ldots,e-1, 
\]
where the subscript $i$ is taken modulo $e$. 
The canonical additive character $\psi_{\F_q}$ of $\F_q$ is defined by 
\[
\psi_{\F_q}(x)=\zeta_p^{\Tr_{q/p}(x)}, \, \, x\in \F_q, 
\]
where $\Tr_{q/p}$ is the absolute trace from $\F_q$ to $\F_p$ defined by 
\[
\Tr_{q/p}(x)=x+x^p+x^{p^2}+\cdots+x^{p^{f-1}}. 
\]  
Define $R_0=\{(x,x)\,:\,x\in \F_q\}$ and $R_i:=\{(x,y)\,:\,x-y \in C_{i-1}^{(e,q)}\}$, $i=1,2,\ldots,e$. Then, $(\F_{q},\{R_i\}_{i=0}^e)$ is  an $e$-class 
translation association scheme, called the {\it cyclotomic scheme}. The first 
eigenmatrix $P$ of the $e$-class cyclotomic scheme  is given by 
the $(e+1)\times (e+1)$ matrix (with the rows of $P$ arranged in a certain way)
\[
P=\begin{pmatrix} 
1  &k &  k&  \cdots &k  \\
1  & \eta_0 & \eta_1&\cdots &\eta_{e-1} \\
1  & \eta_1 & \eta_2&\cdots &\eta_{0} \\
\vdots  &\vdots &\vdots & \cdots & \vdots \\
1  & \eta_{e-1} & \eta_0&\cdots &\eta_{e-2}
 \end{pmatrix}, 
\]
where $k=\frac{q-1}{e}$ and $\eta_i$, $i=0,1,\ldots,e-1$, are given by 
\[
\eta_i=\sum_{x\in C_i^{(e,q)}}\psi_{\F_q}(x), 
\]
the called {\it Gauss periods} of order $e$ of $\F_q$.  

For a multiplicative character
$\chi$ and the canonical additive character $\psi_{\F_q}$ of $\F_q$, we define the {\it Gauss sum} by
\[
G_q(\chi)=\sum_{x\in \F_q^\ast}\chi(x)\psi_{\F_q}(x) \in \Z[\zeta_{q-1},\zeta_p].
\]
We will use the following basic properties of Gauss sums without preamble. 
\begin{enumerate}
\item[(i)] $G_q(\chi)\overline{G_q(\chi)}=q$ if $\chi$ is nontrivial;
\item[(ii)] $G_q(\chi^{-1})=\chi(-1)\overline{G_q(\chi)}$;
\item[(iii)] $G_q(\chi)=-1$ if $\chi$ is trivial.
\end{enumerate}
For a  nontrivial multiplicative character $\chi$ of $\F_q$ and $x\in \F_q^\ast$, by the orthogonality of characters~\cite[P. 195, (5.16)]{LN97}, it holds that   
\begin{equation}\label{eq:oth_g}
\chi(x)=\frac{\chi(-1)G_q(\chi)}{q}\sum_{a\in \F_q^\ast}\chi^{-1}(a)\psi_{\F_q}(ax). 
\end{equation}
On the other hand, the Gauss period $\psi_{\F_q}(C_i^{(e,q)})$ can be expressed as a linear combination of Gauss sums:
\begin{equation}
\psi_{\F_q}(C_i^{(e,q)})=\frac{1}{e}\sum_{j=0}^{e-1}G_q(\chi^{j})\chi^{-j}(\omega^i), \; 0\le i\le e-1,
\end{equation}
where $\chi$ is a multiplicative character of order $e$ of $\F_q$.  For example, if $e=2$, 
we have
\begin{equation}\label{eq:Gaussquad}
\psi_{\F_q}(C_i^{(2,q)})=\frac{-1+(-1)^iG_q(\eta)}{2},\; \quad i=0,1,
\end{equation}
where $\eta$ is the quadratic character of $\F_q$. In particular, the quadratic 
Gauss sum is explicitly evaluated as follows.  
\begin{theorem}\cite[Theorem~5.15]{LN97} \label{thm:Gauss}
Let $q=p^s$ be a prime power with $p$ a prime and $\eta$ be the quadratic character of $\F_q$. 
Then, 
\begin{equation}\label{eq:Gaussquad1}
G_q(\eta)=\begin{cases}
(-1)^{s-1}q^{1/2},& \text{ if }  p\equiv 1\,(\mod{4}), \\
(-1)^{s-1}\zeta_4^s q^{1/2}, & \text{ if } p\equiv 3\,(\mod{4}). 
\end{cases}
\end{equation}
\end{theorem} 
We will use the following formula on Gauss sums of order $8$ of $\F_{q^2}$. 
\begin{theorem}\cite[Theorem~1.5]{Mv} \label{thm:Gauss8}
Let $q=p^f\equiv 3\,(\mod{8})$ be a prime power with $p$ a prime and  $\chi_8$ be a multiplicative character of $\F_{q^2}$. Furthermore, let $\eta$ be the quadratic character of $\F_q$. Then,  $G_{q^2}(\chi_8)/G_q(\eta)\in \Z[\sqrt{-2}]$. In particular, if $q=a^2+2b^2$ is a proper representation of $q$ with $a,b\in\Z$ and $p\not | a+b\sqrt{-2}$. Then, 
\begin{equation}\label{eq:Gaussquad1}
G_{q^2}(\chi_8)=G_q(\eta)(
a+b\sqrt{-2}), 
\end{equation}
where the signs of $a,b$ are 
ambiguously determined. 
\end{theorem} 
Note that only the case where $q$ is a prime in the theorem above was treated in \cite{Mv}. It is easy to generalize the claim into the case where 
$q$ is a prime power by noting that $\Z[\sqrt{-2}]$ is a unique factorization domain. 
In this paper, we do not need to care about the signs of  $a,b$.

Furthermore, we need the following formula on Gauss sums, the so-called 
{\it Davenport-Hasse lifting formula.} 
\begin{theorem}\label{thm:lift}
{\em (\cite[Theorem~11.5.2]{BEW97})}
Let $\chi'$ be a nontrivial multiplicative character of $\F_{q}$ and 
let $\chi$ be the lift of $\chi'$ to $\F_{q^{d}}$, i.e., $\chi(\alpha)=\chi'(\Norm_{q^{d}/q}(\alpha))$ for $\alpha\in \F_{q^{d}}$, where $d\geq 2$ is an integer. Then 
\[
G_{q^d}(\chi)=(-1)^{d-1}(G_{q}(\chi'))^d. 
\]
\end{theorem}

Next, we define {\it Jacobi sums}. 
For multiplicative characters $\chi_1$  and $\chi_2$  of $\F_q$, define 
\[
J_q(\chi_1,\chi_2)=\sum_{x\in \F_q}\chi_1(x)\chi_2(1-x) \in \Z[\zeta_{q-1}].  
\]
Here, we 
extend the domain of multiplicative characters $\chi$ of $\F_q$ to all elements of $\F_q$ by setting 
$\chi(0)=1$ or $\chi(0)=0$ depending on whether $\chi$ is trivial or not. 
There is the following relationship between Jacobi sums and Gauss sums: 
\begin{equation}\label{eq:Gaussrelation}
J_q(\chi_1,\chi_2)=\frac{G_q(\chi_1)G_q(\chi_2)}{G_q(\chi_1\chi_2)}, 
\end{equation}
where  $\chi_1,\chi_2,\chi_1\chi_2$ are nontrivial.  
In this paper, we will use the following formula on Jacobi sums.  
\begin{theorem}\label{lem:facto}{\em (\cite{St})} 
Let $q=p^f\equiv 1\,(\mod{4})$ be a prime power with $p$ a prime. 
Let $\eta$ be the quadratic character of $\F_q$ and $\chi$ a multiplicative  character of order $4$ of $\F_q$. 
 Then,  $J_q(\eta,\chi)\in \Z[\zeta_4]$. In particular, if $p\equiv 1\,(\mod{4})$ and $q=a^2+b^2$ is a proper representation of $q$ with $a$ an odd integer and $p\not | a+b\zeta_4$. Then, 
\begin{equation}\label{eq:Gaussquad1}
J_q(\eta,\chi)=a+b\zeta_4, 
\end{equation}
where the signs of $a,b$ are 
ambiguously determined. 
If $p\equiv 3\,(\mod{4})$, $f$ is even and  $J_q(\eta,\chi)=-(-q)^{f/2}$.  
\end{theorem}
In this paper, we do not need to care about the signs of $a,b$ in the theorem above. 

\section{Basic construction of $t$-intersection sets}\label{sec:comp}
In this paper, we treat $t$-intersection sets  for  block 
designs obtained from quadratic residues of $\F_q$. 
Let $q$ be an odd prime power and $e$ be a positive integer dividing $q^2-1$  
such that $e/\gcd{(e,q+1)}=2$. Then, the restriction of a  multiplicative character of order $e$ of $\F_{q^2}$ to $\F_q$ is of order $2$. 
Let $H$ be an $e/2$-subset of $\{0,1,\ldots,e-1\}$ such that $H\equiv \{0,1,\ldots,e/2-1\}\,(\mod{e/2})$. 
For a fixed positive integer $\ell$ not divisible by $q+1$, define
\begin{equation}\label{eq:defDDD}
D_{\ell,H}=\{x\in \F_q\,:\,1+x \omega^\ell\in \bigcup_{i \in H}C_i^{(e,q^2)}\}, 
\end{equation}
where $\omega$ is a fixed primitive element of $\F_{q^2}$ such that 
$\omega^i\langle \omega^e\rangle =C_i^{(e,q^2)}$.  We will use the set 
$D_{\ell,H}$ to construct $t$-intersection sets. 
In this section, we are interested in the sizes of $D_{\ell,H}$ and $D_{\ell,H}\cap (C_0^{(2,q)}+s)$, $s\in \F_q$. 

We will use the following lemmas.  
\begin{lemma}\label{lem:charac_ge}
With notations as above, 
let $\chi_e$ be a multiplicative
character of order $e$ of $\F_{q^2}$ and $\eta$ be 
the quadratic character of $\F_{q}$. 
Then, 
\[
\sum_{x\in \F_q}\chi_e(1+\omega^\ell x)=\frac{\chi_e(-1)G_{q^2}(\chi_e)G_q(\eta)}{q}\chi_e(\omega^{-q\ell})\chi_e(\omega^{\ell q}-\omega^{\ell}).  
\]
\end{lemma}
\begin{proof}
By \eqref{eq:oth_g}, we have 
\begin{equation}\label{eq:Ga81}
\sum_{x\in \F_q}\chi_e(1+\omega^\ell x)=
\frac{\chi_e(-1)G_{q^2}(\chi_e)}{q^2}\sum_{a\in \F_{q^2}^\ast}\chi_e^{-1}(a)\psi_{\F_{q^2}}(a)\sum_{x\in \F_q}\psi_{\F_{q^2}}(a \omega^\ell x). 
\end{equation}
Furthermore, 
\[
\sum_{x\in \F_q}\psi_{\F_{q^2}}(a \omega^\ell x)=
\sum_{x\in \F_q}\psi_{\F_q}(x\Tr_{q^2/q}(a \omega^\ell))=
\begin{cases}
0,& \text{ if } \Tr_{q^2/q}(a\omega^\ell)\not=0,\\
q,& \text{ if } \Tr_{q^2/q}(a\omega^\ell)=0,  
\end{cases}
\]
where $\Tr_{q^2/q}$ is the relative trace from $\F_{q^2}$ to $\F_q$, i.e., 
$\Tr_{q^2/q}(x):=x+x^q$ for $x\in \F_{q^2}$. 
It is clear that $\Tr_{q^2/q}(a\omega^\ell)=0$ if and only if $a$ has the form 
$a=y\omega^{-\ell+\frac{q+1}{2}}$ with $y\in \F_q$. Hence, continuing from 
\eqref{eq:Ga81}, we have 
\begin{align}
\sum_{x\in \F_q}\chi_e(1+\omega^\ell x)=&\, 
\frac{\chi_e(-1)G_{q^2}(\chi_e)}{q}\sum_{y\in \F_q^\ast}\chi_e^{-1}(y\omega^{-\ell+\frac{q+1}{2}})\psi_{\F_{q^2}}(y\omega^{-\ell+\frac{q+1}{2}})\nonumber\\
=&\, \frac{\chi_e(-1)G_{q^2}(\chi_e)}{q}\sum_{y\in \F_q^\ast}\chi_e^{-1}(y)\chi_e^{-1}(\omega^{-\ell+\frac{q+1}{2}})\psi_{\F_{q}}(y\Tr_{q^2/q}(\omega^{-\ell+\frac{q+1}{2}})). \label{eq:Ga82ge}
\end{align}
Since $\Tr_{q^2/q}(\omega^{-\ell+\frac{q+1}{2}})=\omega^{-\ell+\frac{q+1}{2}}\omega^{-\ell q}(\omega^{\ell q}-\omega^{\ell})$ is a nonzero element in $\F_q$, continuing from
\eqref{eq:Ga82ge}, we have 
\begin{align*}
&\, \sum_{x\in \F_q}\chi_e(1+\omega^\ell x)\\
=&\,
\frac{\chi_e(-1)G_{q^2}(\chi_e)}{q}\sum_{y\in \F_q^\ast}\chi_e^{-1}(y\Tr_{q^2/q}(\omega^{-\ell+\frac{q+1}{2}}))\psi_{\F_q}(y\Tr_{q^2/q}(\omega^{-\ell+\frac{q+1}{2}}))\chi_e(\omega^{-q\ell})
\chi_e(\omega^{\ell q}-\omega^\ell)\\
=&\,
\frac{\chi_e(-1)G_{q^2}(\chi_e)G_q(\eta)}{q}\chi_e(\omega^{-q\ell})
\chi_e(\omega^{\ell q}-\omega^\ell). 
\end{align*}
This completes the proof of the lemma. \qed \end{proof}

\begin{lemma}\label{lem:charac28}
With notations as in Lemma~\ref{lem:charac_ge}, it holds that 
for any $s\in \F_q$ 
\[
\sum_{x\in \F_q\setminus \{s\}}\chi_{e}(1+\omega^\ell x)\eta(x-s)=\frac{G_{q^2}(\chi_e)G_q(\eta)}{q}\chi_e^{-1}(1+\omega^{q\ell}s)
\chi_e(\omega^{\ell q}-\omega^\ell)-\chi_e(\omega^\ell). 
\]
\end{lemma}
\begin{proof} 
Since ${\chi_e}_{|\F_q}=\eta$,  we have 
\begin{align}
\sum_{x\in \F_q\setminus \{s\}}\chi_e(1+\omega^\ell x)\eta(x-s)=&\, \sum_{y\in \F_q^\ast}\chi_e(1+\omega^\ell (y+s))\chi_e^{-1}(y)\nonumber\\
=&\, \sum_{y\in \F_q^\ast}\chi_e(y^{-1}(1+\omega^\ell s)+\omega^\ell)\nonumber\\
=&\, \sum_{y\in \F_q}\chi_{e}(y(1+\omega^\ell s)+\omega^\ell)-\chi_e(\omega^\ell). \label{eq:chara28}
\end{align}
Similarly to Lemma~\ref{lem:charac_ge}, we have 
\begin{equation}\label{eq:Ga83}
\sum_{y\in \F_q}\chi_e(y(1+\omega^\ell s)+\omega^\ell)=
\frac{\chi_e(-1)G_{q^2}(\chi_e)}{q^2}\sum_{a\in \F_{q^2}^\ast}\chi_e^{-1}(a)\psi_{\F_{q^2}}(a\omega^\ell)\sum_{y\in \F_q}\psi_{\F_{q^2}}(a y(1+\omega^\ell s)). 
\end{equation}
Here, 
\[
\sum_{y\in \F_q}\psi_{\F_{q^2}}(a y(1+\omega^\ell s))=
\sum_{y\in \F_q}\psi_{\F_q}(y\Tr_{q^2/q}(a (1+\omega^\ell s)))=
\begin{cases}
0,& \text{ if } \Tr_{q^2/q}(a (1+\omega^\ell s))\not=0,\\
q,& \text{ if } \Tr_{q^2/q}(a (1+\omega^\ell s))=0. 
\end{cases}
\]
It is clear that $\Tr_{q^2/q}(a (1+\omega^\ell s))=0$ if and only if $a$ has the form 
$a=y (1+\omega^\ell s)^{-1}\omega^\frac{q+1}{2}$ with $y\in \F_q$. 
Continuing from \eqref{eq:Ga83}, we have 
\begin{align}
&\, \sum_{y\in \F_q}\chi_e(y(1+\omega^\ell s)+\omega^\ell)\nonumber\\
=&\, 
\frac{\chi_e(-1)G_{q^2}(\chi_e)}{q}\sum_{y\in \F_q^\ast}\chi_e^{-1}(y (1+\omega^\ell s)^{-1}\omega^\frac{q+1}{2}))\psi_{\F_{q^2}}(y (1+\omega^\ell s)^{-1}\omega^{\ell+\frac{q+1}{2}})\nonumber\\
=&\, \frac{\chi_e(-1)G_{q^2}(\chi_e)}{q}\sum_{y\in \F_q^\ast}\chi_e^{-1}(y (1+\omega^\ell s)^{-1}\omega^\frac{q+1}{2})\psi_{\F_q}(y\Tr_{q^2/q}( (1+\omega^\ell s)^{-1}\omega^{\ell+\frac{q+1}{2}})). \label{eq:Ga82}
\end{align}
Since $\Tr_{q^2/q}((1+\omega^\ell s)^{-1}\omega^{\ell+\frac{q+1}{2}})=-(1+\omega^\ell s)^{-1}\omega^\frac{q+1}{2}(\omega^{\ell q}-\omega^\ell)(1+\omega^\ell s)^{-q}$ is a nonzero element in $\F_q$, continuing from
\eqref{eq:Ga82}, we have 
\begin{align*}
\sum_{y\in \F_q}\chi_e(y(1+\omega^\ell s)+\omega^\ell)=&\,
\frac{\chi_e(-1)G_{q^2}(\chi_e)}{q}\sum_{y\in \F_q^\ast}\chi_e^{-1}(y)\psi_{\F_q}(y)\chi_e(-1)\chi_e^{-1}(1+\omega^{q\ell}s)
\chi_e(\omega^{\ell q}-\omega^\ell)\\
=&\,
\frac{G_{q^2}(\chi_e)G_q(\eta)}{q}\chi_e^{-1}(1+\omega^{q\ell}s)
\chi_e(\omega^{\ell q}-\omega^\ell). 
\end{align*}
This completes the proof of the lemma. 
\qed \end{proof}
\vspace{0.3cm}

We now evaluate the size of $D_{\ell,H}$.  
\begin{proposition}\label{prop:sizeHH}
The sizes of $D_{\ell,H}$  defined in \eqref{eq:defDDD} is given by 
\begin{equation}\label{eq:sizeH}
|D_{\ell,H}|=\frac{q}{2}+\frac{\chi_e(-1)G_q(\eta)}{eq}\sum_{i=0}^{e/2-1} A_{2i+1} G_{q^2}(\chi_e^{2i+1})\chi_e^{-(2i+1)}(\omega^{q\ell})\chi_e^{2i+1}(\omega^{\ell q}-\omega^{\ell}), 
\end{equation}
where $A_i=\sum_{j\in H}\zeta_e^{-ji}$. 
\end{proposition}
\proof
The characteristic function of 
 $\bigcup_{j \in H}C_{j}^{(e,q^2)}$ is  given as 
\begin{equation}\label{eq:deff}
f(x)=\frac{1}{e}\sum_{j\in H}\sum_{i=0}^{e-1}\zeta_e^{-ji}\chi_e^i(x), \quad  x\in \F_{q^2}^\ast. 
\end{equation}
The size of  $D_{\ell,H}$ is expressed as 
\[
|D_{\ell,H}|=\sum_{x\in \F_q}f(1+\omega^\ell x)=\frac{1}{e}\sum_{i=0}^{e-1} A_i\sum_{x\in \F_q}\chi_e^i(1+x\omega^\ell). 
\]
Since $H\equiv \{0,1,\ldots,e/2-1\}\,(\mod{e/2})$, we have $A_i=0$ if $i$ is even with $i\not=0$. Furthermore, by Lemma~\ref{lem:charac_ge}, it holds that 
\[
|D_{\ell,H}|=\frac{q}{2}+\frac{\chi_e(-1)G_q(\eta)}{eq}\sum_{i=0}^{e/2-1} A_{2i+1} G_{q^2}(\chi_e^{2i+1})\chi_e^{-(2i+1)}(\omega^{q\ell})\chi_e^{2i+1}(\omega^{\ell q}-\omega^{\ell}).  
\]
This completes the proof. 
\qed
\vspace{0.3cm}

We give another representation for the size of $D_{\ell,H}$ below. 
\begin{proposition}\label{prop:size_charac}
The sizes of $D_{\ell,H}$ defined in \eqref{eq:defDDD} is given by 
\[
|D_{\ell,H}|=\frac{q}{2}+\frac{\chi_e(-1)G_q(\eta)}{2q}+\frac{\chi_e(-1)G_q(\eta)}{q}\sum_{j\in H}\psi_{\F_{q^2}}(\omega^{q\ell}(\omega^{\ell q}-\omega^{\ell})^{-1}C_j^{(e,q^2)}).  
\]
\end{proposition}
\proof
By the orthogonality of characters, we have
\[
\frac{1}{e}\sum_{i=0}^{e-1} A_i G_{q^2}(\chi_e^{i})\chi_e^{-i}(\omega^{q\ell})\chi_e^{i}(\omega^{\ell q}-\omega^{\ell})=\sum_{j\in H}\psi_{\F_{q^2}}(\omega^{q\ell}(\omega^{\ell q}-\omega^{\ell})^{-1}C_j^{(e,q^2)}). 
\]
Hence, continuing from \eqref{eq:sizeH}, we have 
\[
|D_{\ell,H}|=\frac{q}{2}+\frac{\chi_e(-1)G_q(\eta)}{2q}+\frac{\chi_e(-1)G_q(\eta)}{q}\sum_{j\in H}\psi_{\F_{q^2}}(\omega^{q\ell}(\omega^{\ell q}-\omega^{\ell})^{-1}C_j^{(e,q^2)}). 
\]
This completes the proof. 
\qed
\vspace{0.3cm}

It is remarkable that this proposition implies that the size of $D_{\ell,H}$ 
can be evaluated from one of the nontrivial character values of $\bigcup_{i\in H}C_i^{(e,q^2)}$ 
(equivalently, the eigenvalues of the graph $\Gamma$ on $\F_{q^2}$, where  $(x,y)\in E(\Gamma)$ if and only if $x-y\in \bigcup_{i\in H}C_i^{(e,q^2)})$, the so-called {\it Cayley graph}). 

We now evaluate the sizes of  $D_{\ell,H}\cap (C_0^{(2,q)}+s)$, $s\in \F_q$. 
\begin{proposition}\label{prop:size:int}
The size $N_s$ of $D_{\ell,H}\cap (C_0^{(2,q)}+s)$ is given by 
\begin{align}
N_s=&\, \frac{q-1}{4}-\frac{1}{2e}\sum_{i=0}^{e/2-1}
A_{2i+1}(\chi_e^{2i+1}(1+\omega^\ell s)+\chi_e^{2i+1}(\omega^\ell)) \nonumber\\
&\, \, +\frac{G_q(\eta)}{2eq}\sum_{i=0}^{e/2-1}A_{2i+1}
G_{q^2}(\chi_e^{2i+1})\chi_e^{2i+1}(\omega^{\ell q}-\omega^\ell)
(\chi_e^{-(2i+1)}(1+\omega^{q\ell} s)+\chi_e^{-(2i+1)}(-\omega^{q\ell})). 
\label{eq:sizeNs}
\end{align}
\end{proposition}
\proof 
The characteristic function of $C_0^{(2,q)}$ is  given as 
\[
g(x)=\frac{1}{2}(\eta(x)+1), \quad  x\in \F_q^\ast. 
\]
The size $N_s$ of the set $D_{\ell,H}\cap (C_0^{(2,q)}+s)$ is expressed as 
\[
N_s=\sum_{x\in \F_q\setminus \{s\}}f(1+\omega^\ell x)g(x-s). 
\]
By the definitions of $g(x)$ and $f(x)$,  we have 
\begin{align}
N_s
=&\,\frac{1}{2e}\sum_{x\in \F_q\setminus \{s\}}(\eta(x-s)+1)\Big(e/2+\sum_{i=0}^{e/2-1}A_{2i+1}\chi_e^{2i+1}(1+\omega^\ell x)\Big) 
. \label{eq:com1gene}
\end{align}
Put $N_{s,1}^{(i)}=\sum_{x\in \F_q\setminus \{s\}}\chi_e^i(1+\omega^\ell x)$ and 
$N_{s,2}^{(i)}=\sum_{x\in \F_q\setminus \{s\}}\chi_e^i(1+\omega^\ell x)\eta(x-s)$. 
Then, continuing from \eqref{eq:com1gene}, 
\begin{equation}\label{eq:comN8}
N_s=
\frac{q-1}{4}+\frac{1}{2e}\sum_{i=0}^{e/2-1}A_{2i+1}(N_{s,1}^{(2i+1)}+N_{s,2}^{(2i+1)}). 
\end{equation}
By Lemmas~\ref{lem:charac_ge} and \ref{lem:charac28}, we have for 
odd $i$ 
\begin{align}\label{eq:nnn8}
N_{s,1}^{(i)}+N_{s,2}^{(i)}=
\frac{G_{q^2}(\chi_e^i)G_q(\eta)}{q}\chi_e^i(\omega^{\ell q}-\omega^\ell)
(\chi_e^{-i}(1+\omega^{q\ell} s)+\chi_e^{-i}(-\omega^{q\ell}))-(\chi_e^i(1+\omega^\ell s)+\chi_e^i(\omega^\ell)). 
\end{align}
Then, the assertion of the proposition follows. \qed 
\vspace{0.3cm}

We give another representation for the sizes of  $D_{\ell,H}\cap (C_0^{(2,q)}+s)$, $s\in \F_q$,  below. 

\begin{proposition}\label{prop:size:int2}
The size $N_s$ of $D_{\ell,H}\cap (C_0^{(2,q)}+s)$ is given by 
\begin{align*}
N_s=&\, \frac{q-1}{4}+\frac{1-\xi_s-\xi_\ell}{2} \nonumber\\
&\, \, +\frac{G_q(\eta)}{2q}\left(1+
\sum_{i\in H}\psi_{\F_{q^2}}((\omega^{\ell q}-\omega^\ell)^{-1}(1+\omega^{q\ell} s)C_i^{(e,q^2)})+\sum_{i\in H}\psi_{\F_{q^2}}(-(\omega^{\ell q}-\omega^\ell)^{-1}\omega^{q\ell}C_i^{(e,q^2)})\right), 
\end{align*}
where 
$\xi_s=1$ if $1+\omega^\ell s\in \bigcup_{j\in H}C_j^{(e,q^2)}$ and $0$ otherwise, and $\xi_\ell=1$ if $\omega^\ell\in \bigcup_{j\in H}C_j^{(e,q^2)}$ and $0$ otherwise.  
\end{proposition}
\proof 
By the orthogonality of characters, we have 
\[
\frac{1}{e}\sum_{i=0}^{e-1}A_{i}
G_{q^2}(\chi_e^{i})\chi_e^{i}(\omega^{\ell q}-\omega^\ell)
\chi_e^{-i}(1+\omega^{q\ell} s)=
\sum_{i\in H}\psi_{\F_{q^2}}((\omega^{\ell q}-\omega^\ell)^{-1}(1+\omega^{q\ell} s)C_i^{(e,q^2)}) 
\]
and \[
\frac{1}{e}\sum_{i=0}^{e-1}A_{i}
G_{q^2}(\chi_e^{i})\chi_e^{i}(\omega^{\ell q}-\omega^\ell)
\chi_e^{-i}(-\omega^{q\ell})=
\sum_{i\in H}\psi_{\F_{q^2}}(-(\omega^{\ell q}-\omega^\ell)^{-1}\omega^{q\ell}C_i^{(e,q^2)}). 
\] 
Furthermore, we note that 
\[
\frac{1}{2e}\sum_{i=0}^{e/2-1}A_{2i+1}(\chi_e^i(1+\omega^\ell s)+\chi_e^i(\omega^\ell))=\frac{-1+\xi_s+\xi_\ell}{2}. 
\]
Then, by Proposition~\ref{prop:size:int}, the assertion of the proposition follows.  \qed 
 \vspace{0.3cm}

Similarly to Proposition~\ref{prop:size_charac}, the  proposition above implies that the sizes of $D_{\ell,H}\cap (C_0^{(2,q)}+s)$, $s\in \F_q$, can be expressed in terms of character values of $\bigcup_{i\in H}C_i^{(e,q^2)}$.   
\section{Construction of biregular Hadamard matrices with maximum excess: the case where $4m^2+4m+3$ is a prime power}\label{sec:const1}
\subsection{Construction of biregular Hadamard matrices from quadratic residues of $\F_q$ with $q\equiv 3\,(\mod{4})$}

Let $q\equiv 3\,(\mod{4})$ be a prime power.  
Set $P=\F_q$ and 
\begin{equation}\label{eq:desquad}
{\mathcal B}=\{\{x+a:x \in C_0^{(2,q)}\cup \{0\}\}:a \in \F_q\}. 
\end{equation}
Then, $(P,{\mathcal B})$ is a symmetric $2$-design, 
the so-called {\it Paley design}.  
Let $M$ be a $q\times q$ $(1,-1)$-matrix whose rows and columns are labeled by the elements of $\F_q$ and entries are defined by 
\[
M_{i,j}=\begin{cases}
1,& \text{ if } j-i\in C_0^{(2,q)} \cup \{0\},\\
-1,& \text{ if } j-i\in \F_q\setminus (C_0^{(2,q)} \cup \{0\}). 
\end{cases} 
\]
Define 
\begin{equation}\label{eq:conf}
H=\begin{pmatrix} -1  & {\bf 1}_q^\tra \\ {\bf 1}_q & M \end{pmatrix}. 
\end{equation}
Then, $H$ forms an Hadamard matrix of order $n=q+1$. 
We now transform $H$ to a biregular Hadamard matrix with maximum excess by negating some rows and  columns of $H$. 
\begin{proposition}\label{prop:const1}
Let $q=4m^2+4m+3$ be a prime power and $(P,{\mathcal B})$ be the block design defined  in \eqref{eq:desquad}.  Assume that there is a four-intersection set with parameters $(2m^2+m+2;\{m^2+1,m^2+2,m^2+m+1,m^2+m+2\})$ for $(P,{\mathcal B})$.  
Then, there exists a biregular Hadamard 
matrix $H'$ of order $n=4(m^2+m+1)$ such that $H'{\bf 1}_n$ has entries $2m-2$ and 
$2m+2$.  
\end{proposition}
\proof 
Let $D$ be the assumed four-intersection set and $H$ be the Hadamard matrix defined  in \eqref{eq:conf}. Set 
\[
(\alpha_1,\alpha_2,\alpha_3,\alpha_4)=(m^2+1,m^2+2,m^2+m+1,m^2+m+2). 
\]
Let ${\bf x}$ be the support vector of $D$ in $P$.
Furthermore, let 
$A$ (resp. $A^\perp$) be the $q\times q$-diagonal matrix, 
whose entries are defined by $A_{i,i}=-1$ if $i\in D$ and $1$ otherwise (resp. $A^\perp_{i,i}=-1$ if $i\in D_{\alpha_3}^\perp \cup D_{\alpha_4}^\perp$ and $1$ otherwise). 
Define 
\[
B=\begin{pmatrix} 1  & {\bf 0}_q^\tra \\ {\bf 0}_q & A \end{pmatrix}, \quad 
B^\perp=\begin{pmatrix} 1  & {\bf 0}_q^\tra \\ {\bf 0}_q & A^\perp \end{pmatrix}, 
\]
where ${\bf 0}_q$ is the all-zero vector of length $q$. Then,  
$H'=B^\perp H B$ is the desired Hadamard matrix. 

We now show that $B^\perp H B{\bf 1}_{q+1}$ has either $2m-2$ or $2m+2$ as its entries. By the definitions of $B$, $B^\perp$ and $H$, we have 
\begin{align*}
B^\perp H B{\bf 1}_{q+1}=&\, 
\begin{pmatrix} 1  & {\bf 0}_q^\tra \\ {\bf 0}_q & A^\perp \end{pmatrix}
\begin{pmatrix} -1  & {\bf 1}_q^\tra \\ {\bf 1}_q & M \end{pmatrix}
\begin{pmatrix} 1  & {\bf 0}_q^\tra \\ {\bf 0}_q & A \end{pmatrix}
\begin{pmatrix} 1 \\ {\bf 1}_q \end{pmatrix}\\
=&\, 
\begin{pmatrix} -1  + {\bf 1}_q^\tra  A {\bf 1}_q \\
A^\perp ({\bf 1}_q+MA{\bf 1}_q)
 \end{pmatrix}. 
\end{align*}
Since $|D|=2m^2+m+2$, we have $-1  + {\bf 1}_q^\tra  A {\bf 1}_q=2m-2$. 
Furthermore, since $D$ is a four-intersection set with parameters $(2m^2+m+2;\{m^2+1,m^2+2,m^2+m+1,m^2+m+2\})$ for $(P,{\mathcal B})$, we have 
\[
\left(\frac{1}{2}(M+J_q){\bf x}\right)_i=\begin{cases}
m^2+1,& \text{ if } i\in D_{\alpha_1}^\perp,\\
m^2+2,& \text{ if } i\in D_{\alpha_2}^\perp,\\
m^2+m+1,& \text{ if } i\in D_{\alpha_3}^\perp,  \\
m^2+m+2,& \text{ if } i\in D_{\alpha_4}^\perp, 
\end{cases}
\]
where $D_{\alpha_i}^\perp$ is defined in \eqref{eq:defdual}. 
Hence, 
\[
A^\perp \left({\bf 1}_q+MA{\bf 1}_q\right)_i=
A^\perp \left({\bf 1}_q+M{\bf 1}_q-2M{\bf x}\right)_i=
\begin{cases}
2m+2,& \text{ if } i\in D_{\alpha_1}^\perp \cup D_{\alpha_4}^\perp,\\
2m-2,& \text{ if } i\in D_{\alpha_2}^\perp \cup D_{\alpha_3}^\perp. 
\end{cases}
\]
This completes the proof of the proposition. 
 \qed
\subsection{Construction of four-intersection sets satisfying the condition of Proposition~\ref{prop:const1}}
Let $q=4m^2+4m+3$ be a prime power. 
We will follow notations in Section~\ref{sec:comp}  with $e=8$. 
Let $\chi_8$ be a multiplicative character of order $8$ of $\F_{q^2}$ and 
$\eta$ be the quadratic character of $\F_{q}$. Note that the restriction of $\chi_8$ to $\F_q$ is of order $2$, i.e., ${\chi_8}_{|\F_q}=\eta$.  Assume that 
$\chi_8(\omega)=\zeta_8=\sqrt{2}(1+\zeta_4)/2$. By Theorem~\ref{thm:Gauss8}, there are $\epsilon,\delta\in \{-1,1\}$ such that $G_{q^2}(\chi_8)=G_q(\eta)(\epsilon (2m+1)+\delta \sqrt{-2})$.  

Let $\ell$ be an integer not divisible by $q+1$. 
For any fixed $h'\in \{0,1,2,3\}$, put $h=2h'+(1-\epsilon \delta)/2\in \{0,1,\ldots,7\}$. 
Furthermore, fix $\ell$ so that the following conditions are satisfied: 
\begin{equation}\label{eq:cond2}
\chi_8(\omega^\ell)=\zeta_8^{2-5\epsilon \delta-6h'},
\quad 
\chi_8(\omega^{\ell q}-\omega^\ell)=-\zeta_4^{\delta}.  
\end{equation}
We will see in Remark~\ref{rem:cond2} that such a pair $(h',\ell)\in \{0,1,2,3\}\times \{0,1,\ldots,q^2-2\}$ exists.

\begin{theorem}\label{thm:twoint}
Let $q=4m^2+4m+3$ be a prime power, and 
let $h$ and $\ell$ be integers 
defined as above. 
Write $H:=\{h+i\,(\mod{8})\,:\,i=0,1,2,3\}$ and 
define
\[
D_{\ell,H}=\left\{x\in \F_q\,:\, 1+x\omega^\ell\in \bigcup_{i\in H}C_{i}^{(8,q^2)}\right\}.  
\]
Then, it holds that $|D_{\ell,H}|=2m^2+m+2$ and 
\[
\{|D_{\ell,H}\cap ((C_0^{(2,q)}\cup \{0\})+s)|:s\in \F_q\}= \{m^2+1,m^2+2,m^2+m+1,m^2+m+2\}. 
\] 
\end{theorem}
This theorem implies that $D_{\ell,H}$ is a four-intersection set satisfying the condition of Proposition~\ref{prop:const1}. Then, our main Theorem~\ref{thm:main1} follows. 
\vspace{0.3cm}

{\bf  Proof of Theorem~\ref{thm:twoint}:}\,  
We first evaluate the size of $D_{\ell,H}$. Let $A_i=\sum_{j=0}^3\zeta_8^{-(h+j)i}$. By Proposition~\ref{prop:sizeHH}, we have 
\begin{equation}\label{eq:size88}
|D_{\ell,H}|=\frac{q}{2}-\frac{G_q(\eta)}{8q}\sum_{i=0}^{3} A_{2i+1} G_{q^2}(\chi_8^{2i+1})\chi_8^{-(2i+1)}(\omega^{q\ell})\chi_8^{2i+1}(\omega^{\ell q}-\omega^{\ell}). 
\end{equation}
Here, 
$G_{q^2}(\chi_8)G_q(\eta)=-q(\epsilon (2m+1)+\delta \sqrt{-2})$ by Theorem~\ref{thm:Gauss8}. 
Substituting $G_{q^2}(\chi_8)G_q(\eta)=-q(\epsilon (2m+1)+\delta \sqrt{-2})$, $\chi_8(\omega^\ell)=\zeta_8^{2-5\epsilon \delta-6h'}$ and 
$\chi_8(\omega^{\ell q}-\omega^\ell)=-\zeta_4^{\delta}$ into \eqref{eq:size88}, it is routine to see that $|D_{\ell,h}|=2m^2+m+2$.

Next, we evaluate the size $N_s$ of the set $D_{\ell,H}\cap ((C_0^{(2,q)}\cup \{0\})+s)$. 
By Proposition~\ref{prop:size:int}, 
\begin{align}
N_s=&\, \frac{1}{8}(4+\sum_{i=1,3,5,7}A_i\chi_8^i(1+\omega^\ell s))+\frac{q-1}{4}-\frac{1}{16}\sum_{i=0}^{3}
A_{2i+1}(\chi_8^{2i+1}(1+\omega^\ell s)+\chi_8^{2i+1}(\omega^\ell)) \nonumber\\
&\, \, +\frac{G_q(\eta)}{16q}\sum_{i=0}^{3}A_{2i+1}
G_{q^2}(\chi_8^{2i+1})\chi_8^{2i+1}(\omega^{\ell q}-\omega^\ell)
(\chi_8^{-(2i+1)}(1+\omega^{q\ell} s)-\chi_8^{-(2i+1)}(\omega^{q\ell})). 
\label{eq:comN8}
\end{align}
Substituting  
$G_{q^2}(\chi_8)G_q(\eta)=-q(\epsilon (2m+1)+\delta \sqrt{-2})$, $\chi_8(\omega^\ell)=\zeta_8^{2-5\epsilon \delta-6h'}$, and  
$\chi_8(\omega^{\ell q}-\omega^\ell)=-\zeta_4^{\delta}$ into \eqref{eq:comN8}, it is routine to see that 
\begin{equation}\label{eq:dual8}
N_s=
\begin{cases}
m^2+m+2,& \text{ if $\epsilon \delta=1$ and $\chi_8(1+\omega^\ell s)\in \{\zeta_8^{2h'},\zeta_8^{2h'+3},\zeta_8^{2h'+6}\}$,} 
\\
 & \text{\quad or $\epsilon \delta=-1$ and $\chi_8(1+\omega^\ell s)\in \{\zeta_8^{2h'+1},\zeta_8^{2h'+4},\zeta_8^{2h'+6}\}$,} 
\\
m^2+2,& \text{ if $\epsilon \delta=1$ and $\chi_8(1+\omega^\ell s)= \zeta_8^{2h'+1}$,}\\
 & \text{\quad or $\epsilon \delta=-1$ and $\chi_8(1+\omega^\ell s)= \zeta_8^{2h'+3}$,} 
\\
m^2+1,& \text{ if  $\epsilon \delta=1$ and $\chi_8(1+\omega^\ell s)\in \{\zeta_8^{2h'+2},\zeta_8^{2h'+4},\zeta_8^{2h'+7}\}$, }\\
 & \text{\quad or $\epsilon \delta=-1$ and $\chi_8(1+\omega^\ell s)\in \{\zeta_8^{2h'},\zeta_8^{2h'+2},\zeta_8^{2h'+5}\}$,} 
\\
m^2+m+1,& \text{ if  $\epsilon \delta=1$ and $\chi_8(1+\omega^\ell s)=\zeta_8^{2h'+5}$,}
\\ & \text{\quad or $\epsilon \delta=-1$ and $\chi_8(1+\omega^\ell s)=\zeta_8^{2h'+7}$.} 
\end{cases} 
\end{equation}
Thus,  $N_s\in \{m^2+1,m^2+2,m^2+m+1,m^2+m+2\}$. 
\qed 

\begin{remark}{\em \label{rem:cond2}
In this remark, we show that there exists a pair $(h',\ell)\in \{0,1,2,3\} \times \{0,1,\ldots,q^2-2\}$ satisfying the condition 
\eqref{eq:cond2}, i.e., 
the set 
\[
\left\{(h',\ell):(q+1) \not | \, \ell, \chi_8(\omega^\ell)=\zeta_8^{2-5\epsilon \delta-6h'},\chi_8(\omega^{\ell q}-\omega^\ell)=-\zeta_4^{\delta}\right\}
\]
is nonempty. 
Note that $\omega^{\ell q}-\omega^\ell=-\omega^{-\frac{q+1}{2}}\Tr_{q^2/q}(\omega^{\ell+\frac{q+1}{2}})$
is a square but  not a fourth power in $\F_{q^2}$. Hence,  the condition that $\chi_8(\omega^{\ell q}-\omega^\ell)=-\zeta_4^{\delta}$ is equivalent to that 
\[
\eta(\Tr_{q^2/q}(\omega^{\ell+\frac{q+1}{2}}))=\zeta_4^{\delta+(m^2+m+1)}. 
\] 
On the other hand, the condition $\chi_8(\omega^\ell)=\zeta_8^{2-5\epsilon \delta-6h'}$ for $h'\in \{0,1,2,3\}$ is equivalent to 
that $\omega^\ell\in C_1^{(2,q^2)}$. This is 
valid whenever $\ell$ is odd. 
Then, the condition $(q+1)\not |\,\ell$ is  satisfied. 
Therefore, it is enough to see that each of the sets 
\[
T_i=\left\{\omega^\ell\in C_{1}^{(2,q^2)}\,:\,\T_{q^2/q}(\omega^{\ell+\frac{q+1}{2}})\in C_i^{(2,q)}\right\}, \quad i=0,1,
\]
is nonempty. In \cite[Remark~2]{MS}, it was shown that 
\begin{equation}\label{eq:tisize}
|T_i|=-\frac{q-1}{2q}\sum_{x\in C_1^{(2,q^2)}}\psi_{\F_{q^2}}(x\omega^{\frac{q+1}{2}})+\frac{(q-1)(q^2-1)}{4q}. 
\end{equation}
By \eqref{eq:Gaussquad} and Theorem~\ref{thm:Gauss}, continuing from \eqref{eq:tisize}, we have 
\[
|T_i|=-\frac{q-1}{2q}\left(\frac{-1-G_{q^2}(\eta)}{2}\right)+\frac{(q-1)(q^2-1)}{4q}=\frac{q^2-1}{4}>0. 
\]
Hence, each $T_i$ is nonempty. 
}\end{remark}

\section{Construction of biregular Hadamard matrices with maximum excess: the case where $2m^2+2m+1$ is a prime power}\label{sec:const2}
\subsection{Construction of biregular Hadamard matrices from quadratic residues of $\F_q$ with $q\equiv 1\,(\mod{4})$}\label{subsec:const2}
Let $q\equiv 1\,(\mod{4})$ be a prime power. 
Let $M$ be a $q\times q$ $(0,1,-1)$-matrix whose rows and columns are labeled by the elements of $\F_q$ and entries are defined by 
\[
M_{i,j}=\begin{cases}
0,& \text{ if } j-i=0,\\
1,& \text{ if } j-i\in C_0^{(2,q)},\\
-1,& \text{ if } j-i\in \F_q \setminus (C_0^{(2,q)} \cup \{0\}). 
\end{cases} 
\]
Define $M_1=M+I_q$, $M_2=M-I_q$, and $M_3=-M_1$. Furthermore, 
define 
\begin{equation}\label{eq:hada2}
H=\begin{pmatrix} 
1  &-1 &  {\bf 1}_q^\tra&  {\bf 1}_q^\tra  \\
-1  & -1 & {\bf 1}_q^\tra&-{\bf 1}_q^\tra \\
{\bf 1}_q  &  {\bf 1}_q &M_1 & M_2 \\
{\bf 1}_q  & -{\bf 1}_q&M_2 & M_3  
 \end{pmatrix}. 
\end{equation}
Then, $H$ forms a symmetric Hadamard matrix of order $n=2q+2$. 
We transform $H$ to a biregular Hadamard matrix with maximum excess by 
negating some rows and columns of $H$. 

Put \[
N_1=\begin{pmatrix} M_1 \\ M_2  \end{pmatrix} \, \, \mbox{ and }\,  N_2=\begin{pmatrix}  M_2 \\ M_3   \end{pmatrix}. 
\]
We label the rows of the upper half submatrices (resp. the lower half submatrices) of $N_1$ and $N_2$ by the elements of $\{0\}\times \F_q$ (resp. 
$\{1\}\times \F_q$) 
so that $(0,j)$th rows (resp. $(1,j)$th rows) of $N_1$ and $N_2$ are respectively corresponding to $j$th rows of $M_1$ and $M_2$ (resp. $M_2$ and $M_3$). Furthermore, we can, and do, label the columns of $N=\begin{pmatrix} N_1& N_2  \end{pmatrix} $  in the same way with the rows of $N$ since  $N$ is symmetric. Put $P=\{0,1\}\times \F_q$, and let $(P,{\mathcal B}_i)$, $i=1,2$, be the block designs with incidence matrices $(N_i+J_{2q,q})/2$, $i=1,2$, respectively. 

\begin{proposition}\label{prop:const2}
Let $q=2m^2+2m+1$ be a prime power and $(P,{\mathcal B}_i)$, $i=1,2$, be the block designs defined above.  Assume that there is a $2m^2+m$-subset $D$ of $P$ such  that $|D\cap  (\{0\}\times \F_q)|=m^2$ and $|D\cap (\{1\}\times \F_q)|=m^2+m$, which is a four-intersection set with parameters $(2m^2+m;\{m^2,m^2+1,m^2+m,m^2+m+1\})$ for $(P,{\mathcal B}_1)$  and 
 with parameters $(2m^2+m;\{m^2-1,m^2,m^2+m-1,m^2+m\})$ for $(P,{\mathcal B}_2)$.  
Then, there exists a biregular Hadamard 
matrix $H'$ of order $n=4(m^2+m+1)$ such that $H'{\bf 1}_{n}$ has entries $2m-2$ and 
$2m+2$.  
\end{proposition}
\proof 
Let $H$ be the Hadamard matrix defined  in \eqref{eq:hada2}. Set 
\[
(\alpha_1,\alpha_2,\alpha_3,\alpha_4)=(m^2,m^2+1,m^2+m,m^2+m+1)
\]
and 
\[
(\beta_1,\beta_2,\beta_3,\beta_4)=(m^2-1,m^2,m^2+m-1,m^2+m). 
\]

Let ${\bf x}$ (resp. ${\bf y}$) be the support vector of $D\cap  (\{0\}\times \F_q)$ in $\{0\}\times \F_q$ (resp. $D\cap  (\{1\}\times \F_q)$ in $\{1\}\times \F_q$). 
Furthermore, let 
$A_1$ (resp. $A_1^{\perp}$) be the $q\times q$-diagonal matrix, whose
 rows and columns are labeled by the elements of $\{0\}\times \F_q$ and  diagonal entries are defined by  
 $(A_1)_{i,i}=-1$ if $i\in D\cap (\{0\}\times \F_q)$ and $1$ otherwise 
(resp. $(A_1^\perp)_{i,i}=-1$ if $i\in D_{\alpha_3}^\perp \cup D_{\alpha_4}^\perp$ and $1$ otherwise).  Similarly, 
let 
$A_2$ (resp. $A_2^{\perp}$) be the $q\times q$-diagonal matrix, whose 
rows and columns are labeled by the elements of $\{1\}\times \F_q$ and  diagonal entries are defined by  
$(A_2)_{i,i}=-1$ if $D\cap (\{1\}\times \F_q)$ and $1$ otherwise (resp. $(A_2^\perp)_{i,i}=-1$ if $D_{\beta_3}^\perp \cup D_{\beta_4}^\perp$ and $1$ otherwise). 
Define 
\[
B=\begin{pmatrix} 1 &0 & {\bf 0}_q^\tra& {\bf 0}_q^\tra \\
 0 &1 & {\bf 0}_q^\tra& {\bf 0}_q^\tra \\
 {\bf 0}_q &{\bf 0}_q & A_1& O_q \\
 {\bf 0}_q & {\bf 0}_q &O_q& A_2 \end{pmatrix}, \quad 
B^\perp=\begin{pmatrix} 1 &0 & {\bf 0}_q^\tra& {\bf 0}_q^\tra \\
 0 &1 & {\bf 0}_q^\tra& {\bf 0}_q^\tra \\
 {\bf 0}_q &{\bf 0}_q & A_1^\perp& O_q \\
 {\bf 0}_q & {\bf 0}_q &O_q& A_2^\perp \end{pmatrix}, 
\]
where $O_q$ is the $q\times q$ all-zero matrix. Then,  
$H'=B^\perp H B$ is the desired Hadamard matrix. 

We show that $B^\perp H B{\bf 1}_{2(q+1)}$ has either $2m-2$ or $2m+2$ as its entries. By the definitions of $B$, $B^\perp$ and $H$, we have 
\begin{align*}
B^\perp H B{\bf 1}_{2(q+1)}=&\, 
\begin{pmatrix} 1 &0 & {\bf 0}_q^\tra& {\bf 0}_q^\tra \\
 0 &1 & {\bf 0}_q^\tra& {\bf 0}_q^\tra \\
 {\bf 0}_q &{\bf 0}_q & A_1^\perp& O_q \\
 {\bf 0}_q & {\bf 0}_q &O_q& A_2^\perp \end{pmatrix}
\begin{pmatrix} 
1  &-1 &  {\bf 1}_q^\tra&  {\bf 1}_q^\tra  \\
-1  & -1 & {\bf 1}_q^\tra&-{\bf 1}_q^\tra \\
{\bf 1}_q  &  {\bf 1}_q &M_1 & M_2 \\
{\bf 1}_q  & -{\bf 1}_q&M_2 & M_3  
 \end{pmatrix}
\begin{pmatrix} 1 &0 & {\bf 0}_q^\tra& {\bf 0}_q^\tra \\
 0 &1 & {\bf 0}_q^\tra& {\bf 0}_q^\tra \\
 {\bf 0}_q &{\bf 0}_q & A_1& O_q \\
 {\bf 0}_q & {\bf 0}_q &O_q& A_2 \end{pmatrix}
\begin{pmatrix} 1\\1\\{\bf 1}_q \\ {\bf 1}_q \end{pmatrix}\\
=&\, 
\begin{pmatrix} 
{\bf 1}_q^\tra  A_1 {\bf 1}_q +{\bf 1}_q^\tra  A_2 {\bf 1}_q\\
-2+{\bf 1}_q^\tra  A_1 {\bf 1}_q -{\bf 1}_q^\tra  A_2 {\bf 1}_q\\
A_1^\perp (2{\bf 1}_q+M_1A_1{\bf 1}_q+M_2A_2{\bf 1}_q)\\
A_2^\perp (M_2A_1{\bf 1}_q+M_3A_2{\bf 1}_q) 
 \end{pmatrix}.
\end{align*}
Since $|D\cap  (\{0\}\times \F_q)|=m^2$ and $|D\cap (\{1\}\times \F_q)|=m^2+m$, we have ${\bf 1}_q^\tra  A_1 {\bf 1}_q +{\bf 1}_q^\tra  A_2 {\bf 1}_q=2m+2$ and $-2+{\bf 1}_q^\tra  A_1 {\bf 1}_q -{\bf 1}_q^\tra  A_2 {\bf 1}_q=2m-2$. 
Furthermore, since $D$ is a four-intersection set with parameters $(2m^2+m;\{m^2,m^2+1,m^2+m,m^2+m+1\})$ for $(P,{\mathcal B}_1)$  and 
with parameters $(2m^2+m;\{m^2-1,m^2,m^2+m-1,m^2+m\})$ for $(P,{\mathcal B}_2)$, we have 
\[
\left(\frac{1}{2}((M_1+J_q){\bf x}+(M_2+J_q){\bf y})\right)_i=\begin{cases}
m^2,& \text{ if } i\in D_{\alpha_1}^\perp,\\
m^2+1,& \text{ if } i\in D_{\alpha_2}^\perp,\\
m^2+m,& \text{ if } i\in D_{\alpha_3}^\perp,  \\
m^2+m+1,& \text{ if } i\in D_{\alpha_4}^\perp,
\end{cases}
\]
and
\[
\left(\frac{1}{2}((M_2+J_q){\bf x}+(M_3+J_q){\bf y})\right)_i=\begin{cases}
m^2-1,& \text{ if } i\in D_{\beta_1}^\perp,\\
m^2,& \text{ if } i\in D_{\beta_2}^\perp,\\
m^2+m-1,& \text{ if } i\in D_{\beta_3}^\perp,  \\
m^2+m,& \text{ if } i\in D_{\beta_4}^\perp. 
\end{cases}
\] 
Hence, 
\begin{align*}
\left(A_1^\perp (2{\bf 1}_q+M_1A_1{\bf 1}_q+M_2A_2{\bf 1}_q)\right)_i=&\,
\left(A_1^\perp (2{\bf 1}_q+M_1{\bf 1}_q+M_2{\bf 1}_q-2M_1{\bf x}-2M_2{\bf y})\right)_i\\
=&\,
\begin{cases}
2m+2,& \text{ if } i\in D_{\alpha_1}^\perp \cup D_{\alpha_4}^\perp,\\
2m-2,& \text{ if } i\in D_{\alpha_2}^\perp \cup D_{\alpha_3}^\perp, 
\end{cases}
\end{align*}
and 
\begin{align*}
\left(A_2^\perp (M_2A_1{\bf 1}_q+M_3A_2{\bf 1}_q)\right)_i=&\,
\left(A_2^\perp (M_2{\bf 1}_q+M_3{\bf 1}_q-2M_2{\bf x}-2M_3{\bf y})\right)_i\\
=&\,
\begin{cases}
2m+2,& \text{ if } i\in D_{\beta_1}^\perp \cup D_{\beta_4}^\perp,\\
2m-2,& \text{ if } i\in D_{\beta_2}^\perp \cup D_{\beta_3}^\perp. 
\end{cases}
\end{align*}
This completes the proof of the proposition. 
 \qed
\vspace{0.3cm}

Similarly to Proposition~\ref{prop:const2}, we have the following proposition. 
\begin{proposition}\label{prop:const3}
Let $q=2m^2+2m+1$ be a prime power and $(P,{\mathcal B}_i)$, $i=1,2$, be the block designs defined as in Proposition~\ref{prop:const2}.  Assume that there is a $2m^2+m+1$-subset $D$ of $P$ such  that $|D\cap  (\{0\}\times \F_q)|=m^2$ and $|D\cap (\{1\}\times \F_q)|=m^2+m+1$, which is a four-intersection set  with parameters $(2m^2+m+1;\{m^2,m^2+1,m^2+m+1,m^2+m+2\})$ for $(P,{\mathcal B}_1)$ and 
with parameters $(2m^2+m+1;\{m^2-1,m^2,m^2+m,m^2+m+1\})$ for $(P,{\mathcal B}_2)$.  
Then, there exists a biregular Hadamard 
matrix $H'$ of order $n=4(m^2+m+1)$ such that $H'{\bf 1}_{n}$ has entries $2m$ and 
$2m+4$.  
\end{proposition}
We omit the proof of the proposition above since the proof is similar to that of Proposition~\ref{prop:const2}. 
\subsection{Construction of four-intersection sets satisfying the conditions of Propositions~\ref{prop:const2} and \ref{prop:const3}}
Let $q=2m^2+2m+1$ be a prime power. We will follow  notations in Section~\ref{sec:comp} with $e=4$. 
Let $\chi_4$ be a multiplicative character of order $4$ of $\F_{q^2}$, and 
$\eta$ 
the quadratic character of  $\F_{q}$.  It is clear that ${\chi_4}_{|\F_q}=\eta$. 
We first give the following lemma. 
\begin{lemma}\label{lem:Gaussconst}
With notations as above, there  are $\epsilon,\delta\in \{1,-1\}$ such that 
\[
G_q(\eta)G_{q^2}(\chi_4)/q=\epsilon m+\delta (m+1)\zeta_4 \mbox{ or } \epsilon (m+1)+\delta m \zeta_4 
\]
depending on whether $m$ is odd or even. 
\end{lemma}
\proof
Let $\chi_4'$ be the multiplicative character of order $4$ of $\F_q$ such that $\chi_4'(\omega^{q+1})=\chi_4(\omega)$. Then, 
by Theorem~\ref{thm:lift}, we have $
G_{q^2}(\chi_4)=-G_q(\chi_4')^2$. 
By \eqref{eq:Gaussrelation}, we have 
\[
G_q(\eta)G_{q^2}(\chi_4)=-G_q(\eta)G_{q}(\chi_4')^2=-\chi_4'(-1)q\frac{G_q(\eta)G_{q}(\chi_4')}{G_{q}(\chi_4'^3)}=
-\chi_4'(-1)qJ_q(\eta,\chi_4'). 
\]
Finally, by Theorem~\ref{lem:facto}, there are $\epsilon,\delta\in \{1,-1\}$ such that 
$-\chi_4'(-1)J_q(\eta,\chi_4')=\epsilon m+\delta (m+1)\zeta_4$ or $-\chi_4'(-1)J_q(\eta,\chi_4')=\epsilon (m+1)+\delta m \zeta_4$ depending on whether $m$ is odd or even. This completes the proof of the lemma. 
\qed
 \vspace{0.3cm}

Let $\ell$ be an integer not divisible by $q+1$. Fix $\ell$ and $h\in \{0,1,2,3\}$ satisfying the following conditions: 
\begin{equation}\label{eq:cond24RR}
\chi_4(\omega^{\ell})=\zeta_4^{3+h},  \quad
\chi_4(\omega^{\ell q}-\omega^{\ell})=\zeta_4^{\delta (1+2h)}. 
\end{equation}
We will see in Remark~\ref{rem:cond24} that such a pair $(h,\ell)\in \{0,1,2,3\}\times \{0,1,\ldots,q^2-2\}$ exists. 

We first treat the case where $m$ is odd.

\begin{theorem}\label{thm:twoint2}
Let $q=2m^2+2m+1$ be a prime power with $m$ odd. Let $h$ and $\ell$ be integers 
defined as above. Set $(H_0,H_1)=(\{h,h+1\},\{h+1,h+2\})$ or 
$(H_0,H_1)=(\{h+1,h+2\},\{h,h+1\})$ depending on whether
$\epsilon \delta=1$ or $-1$. 
Define
\[
D_{\ell,H_d}=\left\{x\in \F_q\,:\, 1+x\omega^\ell\in \bigcup_{j\in H_d}C_j^{(4,q^2)} \right\}, \quad \, \, d=0,1.  
\]
Then,  it holds that $|D_{\ell,H_0}|=m^2$ and $|D_{\ell,H_1}|=m^2+m$. Furthermore, 
\[\{|D_{\ell,H_0}\cap ((C_0^{(2,q)}\cup \{0\})+s)|+|D_{\ell,H_1}\cap (C_0^{(2,q)}+s)|:s\in \F_q\}=\{m^2,m^2+m+1\} 
\] 
and 
\[
\{|D_{\ell,H_0}\cap (C_0^{(2,q)}+s)|+|D_{\ell,H_1}\cap ((\F_q^\ast\setminus C_0^{(2,q)})+s)|:s\in \F_q\}= \{m^2-1,m^2,m^2+m-1,m^2+m\}. 
\]
\end{theorem}
This theorem implies that $(\{0\}\times D_{\ell,H_0}) \cup (\{1\}\times D_{\ell,H_1})$ satisfies the condition of Proposition~\ref{prop:const2}. Then, the assertion for the case where $m$ is odd in Theorem~\ref{thm:main2} follows. 
\vspace{0.3cm}

{\bf  Proof of Theorem~\ref{thm:twoint2}:}\,  
We first evaluate the sizes of $D_{\ell,H_d}$, $d=0,1$. Put $A_{i,d}=\sum_{j\in H_d}\zeta_e^{-ji}$. 
By Proposition~\ref{prop:sizeHH} and Lemma~\ref{lem:Gaussconst}, 
we have  
\begin{align}\label{eq:sizedouble}
|D_{\ell,H_d}|=&\,\frac{q}{2}+\frac{G_q(\eta)}{4q}\sum_{i=0,1} A_{2i+1,d} G_{q^2}(\chi_4^{2i+1})\chi_4^{-(2i+1)}(\omega^{\ell})\chi_4^{2i+1}(\omega^{\ell q}-\omega^{\ell})\nonumber\\
=&\frac{2m^2+2m+1}{2}+\frac{1}{4}A_{1,d}(\epsilon m+\delta (m+1)\zeta_4)\chi_4^{-1}(\omega^{\ell})\chi_4(\omega^{\ell q}-\omega^{\ell})\nonumber\\
&\quad +\frac{1}{4}A_{3,d}(\epsilon m-\delta (m+1)\zeta_4)\chi_4(\omega^{\ell})\chi_4^{-1}(\omega^{\ell q}-\omega^{\ell}).
\end{align}
Substituting $\chi_4(\omega^{\ell})=\zeta_4^{3+h}$ and 
$\chi_4(\omega^{\ell q}-\omega^{\ell})=\zeta_4^{\delta (1+2h)}$ into \eqref{eq:sizedouble}, it is routine to see that  $(|D_{\ell,H_0}|,|D_{\ell,H_1}|)=(m^2,m^2+m)$. 

Next, we evaluate the following: 
\begin{align*}
&M_{1,s}=|D_{\ell,H_0}\cap ((C_0^{(2,q)}\cup \{0\})+s)|+
|D_{\ell,H_1}\cap (C_0^{(2,q)}+s)|,\\
&M_{2,s}=|D_{\ell,H_0}\cap (C_0^{(2,q)}+s)|+
|D_{\ell,H_1}\cap ((\F_q^\ast \setminus C_0^{(2,q)})+s)|. 
\end{align*}
By Proposition~\ref{prop:size:int} and Lemma~\ref{lem:Gaussconst}, we have 
\begin{align}
M_{1,s}
=\frac{q-1}{2}+\frac{1}{8}\sum_{i=1,3}(A_{i,0}+A_{i,1})N_{s}^{(i)}+\frac{1}{4}\sum_{i=1,3}A_{i,0}\chi_4^i(1+\omega^\ell s) +\frac{1}{2}\label{eq:com14}
\end{align}
and 
\begin{align}
M_{2,s}=&|D_{\ell,H_1}|+\frac{1}{8}\sum_{i=1,3}(A_{i,0}-A_{i,1})N_{s}^{(i)}-\frac{1}{4}\sum_{i=1,3}A_{i,1}\chi_4^i(1+\omega^\ell s) -\frac{1}{2},  \label{eq:com24}
\end{align}
where 
\begin{align*}
N_s^{(i)}=
(\epsilon m+\delta(m+1)\zeta_4^i)
\chi_4^i(\omega^{\ell q}-\omega^\ell)
(\chi_4^{-i}(1+\omega^{\ell} s)+\chi_4^{-i}(\omega^{\ell}))
-(\chi_4^i(1+\omega^\ell s)+\chi_4^i(\omega^\ell)). 
\end{align*}
Substituting $\chi_4(\omega^{\ell})=\zeta_4^{3+h}$ and $\chi_4(\omega^{\ell q}-\omega^{\ell})=\zeta_4^{\delta (1+2h)}$ into 
\eqref{eq:com14} and  \eqref{eq:com24}, it is routine to see that 
\begin{equation*}
M_{1,s}=
\begin{cases}
m^2+m+1,& \text{ if $\epsilon \delta=1$ and $\chi_4(1+\omega^\ell s)\in \{\zeta_4^{h},\zeta_4^{h+1}\}$,} \\
&\text{\quad or $\epsilon \delta=-1$ and 
 $\chi_4(1+\omega^\ell s)\in \{\zeta_4^{h+1},\zeta_4^{h+2}\}$,} 
\\
m^2,& \text{ if $\epsilon \delta=1$ and $\chi_4(1+\omega^\ell s)\in \{\zeta_4^{h+2},\zeta_4^{h+3}\}$,}\\
& \text{\quad or $\epsilon \delta=-1$ and $\chi_4(1+\omega^\ell s)\in \{\zeta_4^{h},\zeta_4^{h+3}\}$,}
\end{cases} 
\end{equation*}
and 
\begin{equation*}
M_{2,s}=
\begin{cases}
m^2-1,& \text{ if $\epsilon \delta=1$ and  $\chi_4(1+\omega^\ell s)=\zeta_4^{h}$,} 
\\
& \text{\quad or $\epsilon \delta=-1$ and $\chi_4(1+\omega^\ell s)=\zeta_4^{h+2}$,} 
\\
m^2+m-1,&  \text{ if $\epsilon \delta=1$ and $\chi_4(1+\omega^\ell s)=\zeta_4^{h+1}$,} 
\\
&  \text{\quad or $\epsilon \delta=-1$ and $\chi_4(1+\omega^\ell s)=\zeta_4^{h+1}$,} 
\\
m^2+m,&  \text{ if $\epsilon \delta=1$ and $\chi_4(1+\omega^\ell s)=\zeta_4^{h+2}$,} 
\\
 &  \text{\quad  or $\epsilon \delta=-1$ and $\chi_4(1+\omega^\ell s)=\zeta_4^{h}$,} 
\\
m^2,&  \text{ if $\epsilon \delta=1$ and $\chi_4(1+\omega^\ell s)=\zeta_4^{h+3}$,} 
\\
&  \text{\quad or $\epsilon \delta=-1$ and $\chi_4(1+\omega^\ell s)=\zeta_4^{h+3}$.} 
\end{cases} 
\end{equation*}
This completes the proof of the theorem. 
 \qed 
\begin{remark}{\em \label{rem:cond24}
In this remark, we show that there exists a pair $(h,\ell)\in \{0,1,2,3\} \times \{0,1,\ldots,q^2-2\}$ satisfying the condition 
\eqref{eq:cond24RR}, i.e., 
the set 
\[
\left\{(h,\ell):(q+1) \not | \, \ell, \chi_4(\omega^{\ell})=\zeta_4^{3+h}, 
\chi_4(\omega^{\ell q}-\omega^{\ell})=\zeta_4^{\delta (1+2h)} \right\}
\]
is nonempty. We assume that $h$ is even. This is valid whenever  
$\ell$ is odd since $ \chi_4(\omega^{\ell})=\zeta_4^{3+h}$. Then,  
the condition $(q+1)\not |\,\ell$ is   satisfied.
Since  $\omega^{\ell q}-\omega^\ell=-\omega^{-\frac{q+1}{2}}\Tr_{q^2/q}(\omega^{\ell+\frac{q+1}{2}})$
is a nonsquare in $\F_{q^2}$, the condition $\chi_4(\omega^{\ell q}-\omega^{\ell})=\zeta_4^{\delta (1+2h)}$ is reformulated as 
\[
\eta(-\Tr_{q^2/q}(\omega^{\ell+\frac{q+1}{2}}))=\zeta_4^{m^2+m+1+\delta }. 
\]
Therefore, it is enough to see that each of the sets 
\[
T_i=\left\{\omega^\ell\in C_{1}^{(2,q^2)}\,:\,\T_{q^2/q}(\omega^{\ell+\frac{q+1}{2}})\in C_i^{(2,q)}\right\}, \quad i=0,1,
\]
is nonempty. 
Similarly to Remark~\ref{rem:cond2}, by Theorem~\ref{thm:Gauss}, we have 
\begin{align*}
|T_i|=&\,-\frac{q-1}{2q}\sum_{x\in C_1^{(2,q^2)}}\psi_{\F_{q^2}}(x\omega^{\frac{q+1}{2}})+\frac{(q-1)(q^2-1)}{4q}\\
=&\,-\frac{q-1}{2q}\left(\frac{-1+G_{q^2}(\eta)}{2}\right)+\frac{(q-1)(q^2-1)}{4q}=\frac{q^2-1}{4}>0. 
\end{align*}
Hence, each $T_i$ is nonempty. 
}\end{remark}

We next treat the case where $m$ is even. 

\begin{theorem}\label{thm:twoint3}
Let $q=2m^2+2m+1$ be a prime power with $m$ even. 
Let $h$ and $\ell$ be integers 
defined in \eqref{eq:cond24RR}. 
Set $(H_0,H_1)=(\{h,h+1\},\{h+1,h+2\})$ or 
$(H_0,H_1)=(\{h+1,h+2\},\{h,h+1\})$ depending on whether
$\epsilon \delta=1$ or $-1$. 
Define
\[
D_{\ell,H_d}=\left\{x\in \F_q\,:\, 1+x\omega^\ell\in \bigcup_{j\in H_d}C_j^{(4,q^2)} \right\}, \quad \, \, d=0,1.  
\]
Then,  it holds that $|D_{\ell,H_0}|=m^2$ and $|D_{\ell,H_1}|=m^2+m+1$. Furthermore,   
\[\{|D_{\ell,H_0}\cap ((C_0^{(2,q)}\cup \{0\})+s)|+|D_{\ell,H_1}\cap (C_0^{(2,q)}+s)|:s\in \F_q\}=\{m^2,m^2+1,m^2+m+1,m^2+m+2\} 
\] 
and 
\[
\{|D_{\ell,H_0}\cap (C_0^{(2,q)}+s)|+|D_{\ell,H_1}\cap ((\F_q^\ast\setminus C_0^{(2,q)})+s)|:s\in \F_q\}= \{m^2,m^2+m\}. 
\]
\end{theorem}
We omit the proof of  Theorem~\ref{thm:twoint3} since 
the proof is similar to that of Theorem~\ref{thm:twoint2}. 
This theorem implies that $(\{0\}\times D_{\ell,H_0}) \cup (\{1\}\times D_{\ell,H_1})$ is a set satisfying the condition of Proposition~\ref{prop:const3}. Then, the assertion for the case where $m$ is even in  Theorem~\ref{thm:main2} follows. 
\section{Construction of regular Hadamard matrices}\label{sec:const3}
In this section, we prove Theorem~\ref{thm:main3tr}. We restate the theorem below.  
\begin{theorem}\label{thm:main3}
Let  $q=2m^2-1$ be a prime power with $m$ odd. 
Let $X_0=\{0\}$, and assume that there are subsets $X_i$, $i=1,2,3,4$, of $\F_{q^2}$ partitioning $\F_{q^2}^\ast$ satisfying the following conditions: 
\begin{enumerate}
\item[(1)] $X_1=\omega^{2m^2}X_3$ and  $X_2=\omega^{2m^2}X_4$; 
\item[(2)] Each $X_i$, $i=1,2,3,4$, is a union of cosets of $C_0^{(4m^2,q^2)}$; 
\item[(3)] $(\F_{q^2},\{R_i\}_{i=0}^4)$ is a four-class association scheme with 
Table~\ref{tab_1} as its first eigenmatrix. Here, for $i=0,1,2,3,4$, $(x,y)\in  R_i$ if and only if $x-y \in X_i$. Furthermore, the sets $Y_i$, $i=0,1,2,3,4$, are defined by $Y_0=\{0\}$, 
$Y_1=\omega^{-m^2\tau}X_1^q$,  $Y_2=\omega^{-m^2\tau}X_2^q$
$Y_3=\omega^{-m^2\tau}X_3^q$, $Y_4=\omega^{-m^2\tau}X_4^q$, 
where $\tau=1$ or $-1$ and $X_i^q:=\{x^q:x\in X_i\}$, $i=1,2,3,4$. 
\end{enumerate}
Then, there exists a regular Hadamard 
matrix of order $n=4m^2$. 
{\footnotesize
\begin{table}[!h]
\begin{center}
\caption{\label{tab_1}
The first eigenmatrix  of $(\F_{q^2},\{R_i\}_{i=0}^4)$:  
$\psi_{\F_{q^2}}(aX_i)$, $a\in \F_{q^2}^\ast$, take four character values depending on whether $a\in Y_i$, $i=1,2,3,4$.
}
\begin{tabular}{|c||c|c|c|c|c|}\hline
   &$X_0$&$X_1$& $X_2$ &$X_3$&$X_4$\\  \hline \hline
$a\in Y_0$&$1$&$m(m^2-1)(m-1)$ &$m(m^2-1)(m+1)$ & $m(m^2-1)(m-1)$ & $m(m^2-1)(m+1)$\\\hline
 $a\in Y_{1}$&$1$  &$\frac{m^2+m-1}{2}-\frac{m}{2}G_q(\eta)$ &  $\frac{-m^2-m}{2}-\frac{m+1}{2}G_q(\eta)$ & $\frac{m^2+m-1}{2}+\frac{m}{2}G_q(\eta)$ &  $\frac{-m^2-m}{2}+\frac{m+1}{2}G_q(\eta)$\\\hline
 $a\in Y_2$&$1$& $\frac{-m^2+m}{2}-\frac{m-1}{2}G_q(\eta)$& $\frac{m^2-m-1}{2}+\frac{m}{2}G_q(\eta)$ &  $\frac{-m^2+m}{2}+\frac{m-1}{2}G_q(\eta)$& $\frac{m^2-m-1}{2}-\frac{m}{2}G_q(\eta)$\\\hline
 $a\in Y_{3}$&$1$& $\frac{m^2+m-1}{2}+\frac{m}{2}G_q(\eta)$ &  $\frac{-m^2-m}{2}+\frac{m+1}{2}G_q(\eta)$ & $\frac{m^2+m-1}{2}-\frac{m}{2}G_q(\eta)$ &  $\frac{-m^2-m}{2}-\frac{m+1}{2}G_q(\eta)$\\\hline
$a\in Y_4$&$1$ &$\frac{-m^2+m}{2}+\frac{m-1}{2}G_q(\eta)$ & $\frac{m^2-m-1}{2}-\frac{m}{2}G_q(\eta)$ & $\frac{-m^2+m}{2}-\frac{m-1}{2}G_q(\eta)$ & $\frac{m^2-m-1}{2}+\frac{m}{2}G_q(\eta)$\\\hline
\end{tabular}
\end{center}
\end{table}}
\end{theorem}

\begin{remark}{\em 
We remark that $W_1=X_1\cup X_3$ (resp. $W_2=X_2\cup X_4$) is a union of cosets of $\F_{q}^\ast$ and it takes exactly two nontrivial  
character values $m^2+m-1$ and $-m^2+m$ (resp. $m^2-m-1$ and $-m^2-m$). Equivalently,  $(\F_{q^2},\{R_i'\}_{i=0}^2)$ is a two-class association scheme, where $R_0':=\{(x,x)\,:\,x\in \F_{q^2}\}$ and 
$R'_i$, $i=1,2$, are defined by $(x,y)\in R_{i}'$ if and only if $x-y\in W_i$. 
On the other hand, it is known that any union of cosets of $\F_{q}^\ast$ in $\F_{q^2}^\ast$  takes exactly two nontrivial character values~\cite[Theorem~2]{BWX}, and there exist such subsets $W_1$ and $W_2$. Then, by Proposition~\ref{prop:BM}, $(\F_{q^2},\{R_i\}_{i=0}^4)$ is a fission scheme of a known two-class association scheme.  
}
\end{remark}
\subsection{Construction of regular Hadamard matrices from quadratic residues of $\F_q$ with $q\equiv 1\,(\mod{4})$}
We will modify the construction given in Proposition~\ref{prop:const2}. Let $M_1,M_2,M_3$ be matrices defined as in Subsection~\ref{subsec:const2}. 

Define
\begin{equation}\label{eq:hada3}
H=\begin{pmatrix} 
1  &1 &  {\bf 1}_q^\tra&  {\bf 1}_q^\tra  \\
1  & -1 & -{\bf 1}_q^\tra&{\bf 1}_q^\tra \\
{\bf 1}_q  &  -{\bf 1}_q &M_1 & M_2 \\
{\bf 1}_q  & {\bf 1}_q&M_2 & M_3  
 \end{pmatrix}, 
\end{equation}
which is the matrix obtained by 
negating the second row and the second  column of the matrix defined in \eqref{eq:hada2}. 

Consider the symmetric submatrix 
\begin{equation}\label{eq:hada4}
N':=\begin{pmatrix} 
 -1 & -{\bf 1}_q^\tra & {\bf 1}_q^\tra \\
 -{\bf 1}_q & M_1& M_2 \\
{\bf 1}_q& M_2 & M_3  
 \end{pmatrix} 
\end{equation}
of $H$. 
Note that 
$(N'+J_{2q+1})/2$ is the incidence matrix of a symmetric $2$-design $(P,{\mathcal B})$ with parameters $(v,k,\lambda)=(2q+1,q,(q-1)/2)$. 

\begin{proposition}\label{prop:const5}
Let $q=2m^2-1$ be a prime power with $m$ odd and $(P,{\mathcal B})$ be the block design defined as above.  Assume that there is a $2m^2-m$-subset $D$ of $P$, 
which is a two-intersection set with parameters $(2m^2-m;\{m^2-m,m^2\})$ for $(P,{\mathcal B})$.  
Then, there exists a regular Hadamard 
matrix of order $n=4m^2$.  
\end{proposition}
\proof 
Let $H$ be the Hadamard matrix defined  in \eqref{eq:hada3}. Set 
\[
(\alpha,\beta)=(m^2-m,m^2). 
\]
Let ${\bf x}$ be the support vector of $D$ in $P$. 
Furthermore, let 
$A$ (resp. $A^{\perp}$) be the $(2q+1)\times (2q+1)$-diagonal matrix, whose 
entries are defined by $A_{i,i}=-1$ if $i\in D$ and $1$ otherwise (resp. $A^\perp_{i,i}=-1$ if $D_{\beta}^\perp$ and  $1$ otherwise). 
Define 
\[
B=\begin{pmatrix} 1 & {\bf 0}_q^\tra \\
  {\bf 0}_q & A \end{pmatrix}, \quad 
B^\perp=\begin{pmatrix} 1 & {\bf 0}_q^\tra \\
  {\bf 0}_q & A^\perp \end{pmatrix}.  
\]
Then,  
$H'=B^\perp H B$ is the desired Hadamard matrix. 

We show that $B^\perp H B{\bf 1}_{2(q+1)}=2m {\bf 1}_{2(q+1)}$. By the definitions of $B$, $B^\perp$ and $H$, we have 
\begin{align*}
B^\perp H B{\bf 1}_{2(q+1)}=&\, 
\begin{pmatrix} 1 & {\bf 0}_{2q+1}^\tra \\
  {\bf 0}_{2+1} & A^\perp \end{pmatrix}
\begin{pmatrix} 1 & {\bf 1}_{2q+1}^\tra \\
  {\bf 1}_{2q+1} & N' \end{pmatrix}
\begin{pmatrix} 1 & {\bf 0}_{2q+1}^\tra \\
  {\bf 0}_{2q+1} & A \end{pmatrix}
\begin{pmatrix} 1 \\ {\bf 1}_{2q+1} \end{pmatrix}\\
=&\, 
\begin{pmatrix} 
1+{\bf 1}_{2q+1}^\tra  A {\bf 1}_{2q+1}\\
A^\perp({\bf 1}_{2q+1}+N'A{\bf 1}_{2q+1})
 \end{pmatrix}.
\end{align*}
Since $D=|2m^2-m|$, we have $1+{\bf 1}_{2q+1}^\tra  A {\bf 1}_{2q+1}=2m$. 
Furthermore, since $D$ is a two-intersection set with parameters $(2m^2-m;\{m^2-m,m^2\})$ for $(P,{\mathcal B})$, we have 
\[
\left(\frac{1}{2}(N'+J_{2q+1}){\bf x}\right)_i=\begin{cases}
m^2-m,& \text{ if } i\in D_{\alpha}^\perp,\\
m^2,& \text{ if } i\in D_{\beta}^\perp. 
\end{cases}
\]
Hence, 
\begin{align*}
\left(A^\perp ({\bf 1}_{2q+1}+N'A{\bf 1}_{2q+1})\right)_i=&\,
\left(A^\perp ({\bf 1}_{2q+1}+N'{\bf 1}_{2q+1}-2N'{\bf x})\right)_i
=2m. 
\end{align*}
This completes the proof of the proposition. 
 \qed

\subsection{Construction of two-intersection sets satisfying the conditions of Proposition~\ref{prop:const5}}
We  construct  two-intersection sets  satisfying the conditions of Proposition~\ref{prop:const5} under the assumption of the existence of 
a four-class association scheme. 

With notations as in Theorem~\ref{thm:main3}, assume that there is an integer $\ell$ not divisible by $q+1$ such that 
\begin{equation}\label{eq:lext}
 \omega^{\ell}\in X_2\cup X_4,\quad \omega^{\ell q}-\omega^\ell\in C_{\tau m^2}^{(4m^2,q^2)}. 
\end{equation}
We will see in Remark~\ref{rem:cond25} that such an element $\ell \in \{0,1,\ldots,q^2-2\}$ exists. 
Then, by the definition of $Y_i$, we have $\omega^{q\ell}(\omega^{\ell q}-\omega^\ell)^{-1}\in Y_2$ or $Y_4$ depending on whether $\omega^{\ell}\in X_2$ or $X_4$. Furthermore,  
$(\omega^{\ell q}-\omega^\ell)^{-1}(1+\omega^{q\ell}s)\in Y_i$ if and only if 
$1+\omega^\ell s \in X_i$ for every $i=1,2,3,4$. 
\begin{theorem}\label{thm:twoint4}
Let $q=2m^2-1$ be a prime power with $m$ odd. Assume that there are subsets $X_i$, $i=1,2,3,4$, satisfying the conditions of Theorem~\ref{thm:main3}. 
Let $(S_0,S_1)=(X_1\cup X_4,X_1\cup X_2)$ or 
$(S_0,S_1)=(X_2\cup X_3,X_3\cup X_4)$ depending on whether $\omega^\ell\in X_2$ or $\omega^\ell\in X_4$. 
Define
\[
D_{\ell,S_d}=\left\{x\in \F_q\,:\, 1+x\omega^\ell\in S_d
 \right\}, \quad \, \, d=0,1.  
\]
Then,  it holds that $|D_{\ell,S_0}|=m^2-m$ and $|D_{\ell,S_1}|=m^2$. Furthermore, 
\[\{|D_{\ell,S_0}\cap ((C_0^{(2,q)}\cup \{0\})+s)|+|D_{\ell,S_1}\cap (C_0^{(2,q)}+s)|:s\in \F_q\}=\{m^2-m,m^2\}
\] 
and 
\[
\{|D_{\ell,S_0}\cap (C_0^{(2,q)}+s)|+|D_{\ell,S_1}\cap ((\F_q^\ast\setminus C_0^{(2,q)})+s)|:s\in \F_q\}=\{m^2-m,m^2\}. 
\]
\end{theorem}
We use the same labeling of rows and columns of the matrix  \[
N=\begin{pmatrix} M_1 & M_2 \\ 
M_2 & M_3   \end{pmatrix},  
\]
which is defined in Subsection~\ref{subsec:const2}.  
The theorem above implies that $(\{0\}\times D_{\ell,S_0}) \cup (\{1\}\times D_{\ell,S_1})\subset P$ is the desired two-intersection set for the block design $(P,{\mathcal B})$ in Proposition~\ref{prop:const5}. (So, we do not negate the second row and column of $H$.) Theorem~\ref{thm:main3} then follows.  
\vspace{0.3cm}

{\bf  Proof of Theorem~\ref{thm:twoint4}:}\,  
We consider the case where $\omega^\ell \in X_2$. 
Put $S_0=X_1\cup X_4$ and $S_1=X_1\cup X_2$. In this case, 
$\omega^{q\ell}(\omega^{\ell q}-\omega^{\ell})^{-1}\in Y_2$. 

We first evaluate the sizes of $D_{\ell,S_d}$, $d=0,1$. 
By Proposition~\ref{prop:size_charac},  
\begin{equation}\label{eq:substi}
|D_{\ell,S_d}|=\frac{2m^2-1}{2}+\frac{G_q(\eta)}{2q}(1+2\psi_{\F_{q^2}}(\omega^{q\ell}(\omega^{\ell q}-\omega^{\ell})^{-1}S_d)).  
\end{equation}
By substituting the character values of $S_0$ and $S_1$ 
listed in Table~\ref{tab_1} into \eqref{eq:substi}, we have 
$(|D_{\ell,S_0}|,|D_{\ell,S_1}|)=(m^2-m,m^2)$. Here, we used the fact that $G_q(\eta)^2=q$. 

Next, we evaluate the following: 
\begin{align*}
&M_{1,s}=|D_{\ell,S_0}\cap ((C_0^{(2,q)}\cup \{0\})+s)|+
|D_{\ell,S_1}\cap (C_0^{(2,q)}+s)|,\\
&M_{2,s}=|D_{\ell,S_0}\cap (C_0^{(2,q)}+s)|+
|D_{\ell,S_1}\cap ((\F_q^\ast \setminus C_0^{(2,q)})+s)|. 
\end{align*}
By Proposition~\ref{prop:size:int2}, we have 
\begin{align}
M_{1,s}
=&\,\frac{G_q(\eta)}{2q}\left(2+\psi_{\F_{q^2}}((\omega^{\ell q}-\omega^\ell)^{-1}(1+\omega^{q\ell} s)(S_0\cup S_1))\right.\nonumber\\
&\hspace{1.0cm}\left.+\psi_{\F_{q^2}}(-(\omega^{\ell q}-\omega^\ell)^{-1}\omega^{q\ell}(S_0\cup S_1))\right)
+\frac{q+1}{2}+\frac{\xi_s-\xi_{\ell}-\xi_s'-\xi_{\ell}'}{2}, 
\label{eq:com14R}
\end{align}
where $\xi_s$, $\xi_s'$, $\xi_\ell$, and $\xi_\ell'$ are defined 
as
\[
\xi_s=\begin{cases}
1,& \text{ if } 1+\omega^\ell s\in S_0,\\
0,& \text{ otherwise, }
\end{cases} \quad
\xi_s'=\begin{cases}
1,& \text{ if } 1+\omega^\ell s\in S_1,\\
0,& \text{ otherwise, }
\end{cases}
\]
\[
\xi_\ell=\begin{cases}
1,& \text{ if } \omega^\ell \in S_0,\\
0,& \text{ otherwise, }
\end{cases} \quad
\xi_\ell'=\begin{cases}
1,& \text{ if } \omega^\ell \in S_1,\\
0,& \text{ otherwise. }
\end{cases}
\] 
By Table~\ref{tab_1}, we have 
\begin{align*}
1+\psi_{\F_{q^2}}((\omega^{\ell q}-\omega^\ell)^{-1}(1+\omega^{q\ell} s)(S_0\cup S_1))
= \begin{cases}
-mG_{q}(\eta),& \text{ if } 1+\omega^{\ell} s\in X_1,\\
-(m-1) G_{q}(\eta),& \text{ if } 1+\omega^{\ell} s\in X_2, \\
m G_{q}(\eta),& \text{ if }1+\omega^{\ell} s\in X_3,\\
(m-1) G_{q}(\eta),& \text{ if } 1+\omega^{\ell} s\in X_4, 
\end{cases}
\end{align*}
and 
\[
1+\psi_{\F_{q^2}}(-(\omega^{\ell q}-\omega^\ell)^{-1}\omega^{q\ell}(S_0\cup S_1))=-(m-1)G_q(\eta). 
\]
Furthermore, 
\[
\frac{\xi_s-\xi_{\ell}-\xi_s'-\xi_{\ell}'}{2}=\begin{cases}
-\frac{1}{2},& \text{ if } 1+\omega^{\ell} s\in X_1,\\
-1,& \text{ if } 1+\omega^{\ell} s\in X_2, \\
-\frac{1}{2},& \text{ if }1+\omega^{\ell} s\in X_3,\\
0,& \text{ if } 1+\omega^{\ell} s\in X_4. 
\end{cases}
\]
Hence, we obtain 
\[
M_{1,s}=\begin{cases}
m^2-m,& \text{ if } 1+\omega^\ell s\in X_1 \cup X_2,\\
m^2,& \text{ if } 1+\omega^\ell s\in X_3 \cup X_4. 
\end{cases}
\]
Next, we evaluate $M_{2,s}$.  By Proposition~\ref{prop:size:int2}, we have 
\begin{align}
M_{2,s}=&|D_{\ell,S_1}|+\frac{G_q(\eta)}{2q}\left(\psi_{\F_{q^2}}((\omega^{\ell q}-\omega^\ell)^{-1}(1+\omega^{q\ell} s)S_0)-\psi_{\F_{q^2}}((\omega^{\ell q}-\omega^\ell)^{-1}(1+\omega^{q\ell} s)S_1)\right.\nonumber\\
&\hspace{1.0cm}\left.+\psi_{\F_{q^2}}(-(\omega^{\ell q}-\omega^\ell)^{-1}\omega^{q\ell}S_0)-\psi_{\F_{q^2}}(-(\omega^{\ell q}-\omega^\ell)^{-1}\omega^{q\ell}S_1)\right)
+\frac{-\xi_s-\xi_{\ell}-\xi_s'+\xi_{\ell}'}{2}. 
\label{eq:com14R2}
\end{align}
By Table~\ref{tab_1}, we have 
\begin{align*}
&\psi_{\F_{q^2}}((\omega^{\ell q}-\omega^\ell)^{-1}(1+\omega^{q\ell} s)S_0)-\psi_{\F_{q^2}}((\omega^{\ell q}-\omega^\ell)^{-1}(1+\omega^{q\ell} s)S_1)\\=&\, \begin{cases}
(m+1)G_{q}(\eta),& \text{ if } 1+\omega^{\ell} s\in X_1,\\
-m G_{q}(\eta),& \text{ if } 1+\omega^{\ell} s\in X_2, \\
-(m+1) G_{q}(\eta),& \text{ if }1+\omega^{\ell} s\in X_3,\\
m G_{q}(\eta),& \text{ if } 1+\omega^{\ell} s\in X_4, 
\end{cases}
\end{align*}
and 
\[
\psi_{\F_{q^2}}(-(\omega^{\ell q}-\omega^\ell)^{-1}\omega^{q\ell}S_0)-\psi_{\F_{q^2}}(-(\omega^{\ell q}-\omega^\ell)^{-1}\omega^{q\ell}S_1)=-mG_q(\eta). 
\]
Furthermore, 
\[
\frac{-\xi_s-\xi_{\ell}-\xi_s'+\xi_{\ell}'}{2}=\begin{cases}
-\frac{1}{2},& \text{ if } 1+\omega^{\ell} s\in X_1,\\
0,& \text{ if } 1+\omega^{\ell} s\in X_2, \\
\frac{1}{2},& \text{ if }1+\omega^{\ell} s\in X_3,\\
0,& \text{ if } 1+\omega^{\ell} s\in X_4. 
\end{cases}
\]
Hence, we obtain 
\[
M_{2,s}=\begin{cases}
m^2,& \text{ if } 1+\omega^\ell s\in X_1 \cup X_4,\\
m^2-m,& \text{ if } 1+\omega^\ell s\in X_2 \cup X_3. 
\end{cases}
\]
This completes the proof of the theorem for the case where $\omega^\ell \in X_2$. The proof for the case where $\omega^\ell \in X_4$ is similar. 
 \qed 

\begin{remark}{\em \label{rem:cond25}
In this remark, we show that there exists $\ell\in \{0,1,\ldots,q^2-2\}$  satisfying the condition 
\eqref{eq:lext}, i.e., 
the set 
\[
S=\left\{\ell :(q+1) \not | \, \ell, \omega^{\ell}\in X_2\cup X_4, 
\omega^{\ell q}-\omega^{\ell}\in C_{\tau m^2}^{(4m^2,q^2)} \right\}
\]
is nonempty. Since $|X_2\cup X_4|>q$, there is $\omega^\ell \in X_2\cup X_4$ such that $\omega^\ell \not \in \F_q$.  Furthermore, note that $X_2\cup X_4$ is a union of cosets of $\F_{q}^\ast$. Hence, we evaluate the size of the set 
\[
T=\left\{\ell: \omega^{\ell}\in \omega^t \F_q^\ast, 
\omega^{\ell q}-\omega^{\ell}\in C_{\tau m^2}^{(4m^2,q^2)} \right\}, 
\]
where $\omega^t\not \in \F_q$. 
Since  $\omega^{\ell q}-\omega^\ell=-\omega^{-\frac{q+1}{2}}\Tr_{q^2/q}(\omega^{\ell+\frac{q+1}{2}})$, 
the condition $\omega^{\ell q}-\omega^{\ell}\in C_{\tau m^2}^{(4m^2,q^2)}$ is reformulated as 
\[
\Tr_{q^2/q}(\omega^{\ell+\frac{q+1}{2}})\in C_{(\tau +1)m^2}^{(4m^2,q^2)}=
C_{(\tau +1)/2}^{(2,q)}. 
\]
Hence, we have  
\[
T=\left\{\ell:\omega^\ell \in \omega^t \F_q^\ast,\,\T_{q^2/q}(\omega^{\ell+\frac{q+1}{2}})\in C_{(\tau +1)/2}^{(2,q)}\right\}. 
\]
The size of $T$ is given by 
\begin{align*}
|T|=&\, \frac{1}{q}\sum_{a\in \F_q}\sum_{x\in \omega^t \F_q^\ast}\sum_{b\in C_{(\tau +1)/2}^{(2,q)}} \psi_{\F_{q^2}}(ax\omega^{\frac{q+1}{2}})\psi_{\F_q}(-ab)\\
=&\, \frac{1}{q}\sum_{a\in \F_q^\ast}\sum_{x\in  \F_q^\ast}\sum_{b\in C_{(\tau +1)/2}^{(2,q)}} \psi_{\F_{q^2}}(x\omega^{t+\frac{q+1}{2}})\psi_{\F_q}(-ab)+\frac{(q-1)^2}{2q}. 
\end{align*}
Since $\Tr_{q^2/q}(\omega^{t+\frac{q+1}{2}})\not=0$, we have 
\[
|T|=\frac{q-1}{2q}\sum_{a\in \F_q^\ast}\sum_{x\in  \F_q^\ast} \psi_{\F_{q}}(x\Tr_{q^2/q}(\omega^{t+\frac{q+1}{2}}))\psi_{\F_q}(a)+\frac{(q-1)^2}{2q}
=\frac{q-1}{2}>0. 
\]
Hence, $T$ is nonempty, which also implies that $S$ is nonempty. 
}\end{remark}
\begin{example}{\em 
In the cases where $m=3,5$, we can find subsets $X_i$, $i=1,2,3,4$, satisfying the conditions of Theorem~\ref{thm:main3}. Hence, in these cases, the Hadamard matrix $H$ defined in \eqref{eq:hada3} can be transformed to a  regular Hadamard matrix by negating some rows and columns of $H$. 

In the case where $m=3$, 
define four subsets of $\{0,1,\ldots,11\}$: 
\begin{align*}
&H_1=\{1,5\}, \quad H_2=\{0,2,9,10\}, \\
&H_3=\{7,11\}, \quad H_4=\{3,4,6,8\}, 
\end{align*}
and define $X_i=\bigcup_{j\in H_i}C_j^{(12,17^2)}$, $i=1,2,3,4$. It is clear that $X_1=\omega^{2m^2} X_3$,  $X_2=\omega^{2m^2} X_4$, and each $X_i$ is a union of cosets of $C_j^{(36,17^2)}$. Furthermore, we checked by computer that $X_i$, $i=1,2,3,4$, take character values listed in Table~\ref{tab_1}. 

In the case where $m=5$, 
define four subsets of $\{0,1,\ldots,19\}$: 
\begin{align*}
&H_1=\{2,3,10,19\}, \quad H_2=\{1, 7, 14, 15, 16, 18\}, \\
&H_3=\{0,9,12,13\}, \quad H_4=\{4, 5, 6, 8, 11, 17\}, 
\end{align*}
and define $X_i=\bigcup_{j\in H_i}C_j^{(20,7^4)}$, $i=1,2,3,4$. It is clear that $X_1=\omega^{2m^2} X_3$,  $X_2=\omega^{2m^2} X_4$, and each $X_i$ is a union of cosets of $C_j^{(100,7^4)}$. Furthermore, we checked by computer that $X_i$, $i=1,2,3,4$, take character values listed in Table~\ref{tab_1}. 

In the case where $m=7$, we checked by computer that there is no translation association scheme satisfying the conditions of Theorem~\ref{thm:main3} if  each part $X_i$ is a union of cosets of $C_0^{(28,97^2)}$. 
However, it may be possible if each $X_i$ is a  union of cosets of $C_0^{(196,97^2)}$ not a union of 
cosets of $C_0^{(28,97^2)}$. But, in this case, our computer could not work. 
}\end{example}

\section{Open problems}
We conclude this paper by listing some problems for future works. 

In this paper, we found new constructions of Hadamard matrices with maximum 
excess. In particular, 
we proved that if either of 
$4m^2+4m+1$ or $2m^2+2m+1$ is a prime power, then there exists a 
biregular Hadamard matrix $H$ of order $n=4m^2+4m+4$ with maximum excess attaining the bound of Proposition~\ref{prop:Hadabound}. 
The constructions are based on 
negating some rows and 
columns of known Hadamard matrices obtained from quadratic residues of 
finite fields. 
In a similar way to the biregular case, we tried to find regular Hadamard matrices. 
In particular, we gave a sufficient condition for the matrix $H$ defined in \eqref{eq:hada3} being transformed to a regular Hadamard matrix of order $2(q+1)$ 
under the assumption of the existence of a four-class translation 
association scheme on $\F_{q^2}$, where $q=2m^2-1$ with $m$ odd. Furthermore, we found association schemes satisfying the conditions of Theorem~\ref{thm:main3} in 
the cases where $m=3$ and $5$. So, an open problem remains whether such 
a four-class association scheme exists when $m\ge 7$. 
\begin{problem}
Let $q=2m^2-1$ be a prime power with $m$ an odd integer.  
Does there exist a four-class association scheme  satisfying the conditions of Theorem~\ref{thm:main3} for $m\ge 7$?  
\end{problem}
More generally, we have the following problem. 
\begin{problem}
Can the matrix
$H$ defined in \eqref{eq:hada3} be transformed to a regular Hadamard matrix?  
\end{problem}
In this paper, we did not care about the ``symmetry" of Hadamard matrices.  
So, we give the following problem. 
\begin{problem}
Can the Hadamard matrices 
 obtained in Theorems~\ref{thm:main1}, \ref{thm:main2}  and \ref{thm:main3tr} be transformed to symmetric  Hadamard matrices by permuting  rows or columns?  
\end{problem}


\end{document}